\documentclass[a4paper,17pt]{extarticle}
\usepackage{amsmath, amsthm, amsfonts, amssymb}
\usepackage[abbrev,non-sorted-cites]{amsrefs}
\usepackage{latexsym}
\usepackage{graphicx}
\usepackage{array}
\usepackage{float}
\usepackage{cite, hyperref}
\usepackage{xcolor}
\usepackage{geometry}
\usepackage{wrapfig}
\usepackage[font=footnotesize]{caption}
\usepackage{titlesec}
\usepackage{extarrows}
\titleformat{\section}{\normalsize\bfseries}{\S \thesection}{1em}{}
\titleformat{\subsection}{\normalsize}{\S \thesubsection}{1em}{}

\providecommand{\keywords}[1]
{\small	\textbf{\textbf{Keywords---}} #1
}

\numberwithin{equation}{section}
\theoremstyle{definition}
\newtheorem{defn}{Definition}[section]
\newtheorem{rmk}[defn]{Remark}
\newtheorem{ex}{Example}
\newtheorem{prob}{Problem}

\theoremstyle{plain}
\newtheorem{thm}[defn]{Theorem}
\newtheorem{prop}[defn]{Proposition}

\newtheorem*{thmA}{Theorem A}
\newtheorem*{thmB}{Theorem B}
\newtheorem*{thmS}{Theorem S}
\newtheorem*{thmZ}{Theorem Z}
\newtheorem*{thmJ}{Theorem J}

\newtheorem*{lemA}{Lemma A}
\newtheorem*{lemB}{Lemma B}
\newtheorem*{lemC}{Lemma C}
\newtheorem*{lemD}{Lemma D}

\begin{document}

\title{\large A global Morse index theorem and applications to Jacobi fields on
CMC surfaces}
\author{\small Wu-Hsiung Huang}
\date{\small \today}
\maketitle

\begin{abstract}
In this paper, we establish a ``global" Morse index theorem. Given a hypersurface $M^{n}$ of
constant mean curvature, immersed in $\mathbb{R}^{n+1}$. Consider a
continuous deformation of ``generalized" Lipschitz domain $D(t)$ enlarging in
$M^{n}$. The topological type of $D(t)$ is permitted to change along $t$, so
that $D(t)$ has an arbitrary shape which can ``reach afar" in $M^{n}$, i.e.,
cover any preassigned area. The proof of the global Morse index theorem is
reduced to the continuity in $t$ of the Sobolev space $H_{t}$ of variation
functions on $D(t)$, as well as the continuity of eigenvalues of the
stability operator. We devise a ``detour" strategy by introducing a notion of
``set-continuity" of $D(t)$ in $t$ to yield the required continuities of
$H_{t}$ and of eigenvalues. The global Morse index theorem thus follows and
provides a structural theorem of the existence of Jacobi fields on domains in
$M^{n}$.
\end{abstract}


\newpage%
%%===================================================================

%%========================================================= sec. 0 ==
Consider a parametric hypersurface $M^{n}$ of constant mean curvature immersed
in $\mathbb{R}^{n+1}$ (abbr.~as CMC hypersurface), $n = \dim M^{n}$. For a
nonparametric case, where $u = u(\mathbf{x})$ has constant mean curvature
$H_{o}$, $\mathbf{x} \in \Omega \subset \mathbb{R}^{n}$, Giusti \cite{G78}
called a domain $\Omega$ \emph{extremal} for $u$, if the normal derivative of
$u$ on the boundary $\partial\Omega$ is infinite, and hence the domain of
definition can not be extended further. For example, a unit disk in
$\mathbb{R}^{2}$ is extremal for $H_{o} = 1$, noting that the upper hemisphere
defined on the disk as a graph has constant mean curvature $1$. Huang and Lin
\cite{HL98} extended the notion to a general parametric $M^{n}$, by defining a
domain $D \subset M^{n}$ ``extremal", if the first Dirichlet eigenvalue
$\lambda_{1}$ of the stability operator $L$ is zero, where $L$ is defined on
the space of smooth variation functions with zero boundary condition (see
\eqref{e1.2}).\\

This corresponds to the definition of Giusti in the nonparametric case. An
interesting result is that any negatively curved set $M_{-}$ in $M^{2}$ must be
``large", in the sense that $M_{-}$ contains at least an extremal domain, if
$M_{-} \subset Int \, M^{2}$. Remark that an extremal domain itself is large,
as Giusti's definition indicates. Also, an extremal domain in $M^{n}$ is
stable, and lies in a critical domain on which the first nontrivial Jacobi
field appears. The critical domain is called ``having the first conjugate
boundary" and becomes unstable as it slightly enlarges in $M^{n}$ (see
Theorem~\ref{T1.9}).\\

Given a continuous deformation of domains $D(t)$ enlarging in $M^{n}$, $t \in
[0,b]$, where $D(0)$ is a small neighborhood of a given point $p_o$ in $M^{n}$.
An extremal domain is denoted  in this paper by $D[\lambda_{1} = 0]$. Consider
in $M^{n}$ the larger domains $D[\lambda_{2} = 0]$, $D[\lambda_{3} = 0]$,
$\ldots$, on which a given $\lambda_{k}$ vanishes. Here $\lambda_{k}$ denotes
the $k$th Dirichlet eigenvalue of $L$. Let $t_{k}$ be defined by $D(t_{k}) =
D[\lambda_{k} = 0]$.

\begin{prob} \label{Prob1}
When do the nontrivial Jacobi fields appear on $D(t)$? And how they are
distributed along the $t$-axis, relative to $t_{k}$ as the reference points?
\end{prob}

Following Smale's earlier work \cite{S65}, Frid--Thayer established in
\cite{FT90} an abstract version of the Morse index theorem, which is profound
and interesting. However, they considered smooth deformation of a smooth domain
$D(t)$ enlarging in $M^{n}$. The assumption made the Morse index theorem
staying ``local" and restrictive. The restriction is that any $D(t)$ has smooth
boundary and the deformation is smooth in the sense that the boundary of $D(t)$
depends in a $C^{\infty}$ manner on $t$. In particular, all $D(t)$ are required
diffeomorphic to each other. Thus, for example, a disk $D(t)$ on a given
cylinder should be always of disk-type, and never be enlarged to become a
ring.\\

We treat in this paper a larger class of monotone continuums of domains $D(t)$
enlarging on $M^{n}$ ``continuously" in $t$, so that $D(t)$ may change its
topological type as $t$ increases, and hence possibly extend afar with various
forms. In this sense, we extend the Morse index theorem from local to
``global". Remark that $D(t)$ can be ``generalized" Lipschitz domains (see
Definition~\ref{D2.2}), not necessarily smooth nor homeomorphic to each
other.\\

The proof of the global Morse index theorem is reduced to the \emph{Sobolev
continuity} (see \eqref{e3.1}) as well as the continuity of eigenvalues
$\lambda_{k}(t)$ and $\widetilde{\lambda_{k}}(t)$ of the stability operators
$L$ and $\widetilde{L}$ in $t$ (see \eqref{e1.1}, \eqref{e1.2}, \eqref{e2.9},
\eqref{e2.11}). Remark that the Sobolev spaces do not deformed continuously in
general and the eigenvalues are \emph{not} always continuous as expected. For
example, they are not continuous in $t$, even when $D(t)$ is continuous in $t$
in terms of the Hausdorff distance (see \eqref{e2.2}).

\begin{prob} \label{Prob2}
\quad If the enlarging domains $D(t)$ in $M^{n}$ are allowed to change their
topological types, when the Sobolev continuity is maintained and the
eigenvalues of $L$ and of $\widetilde{L}$ are continuous in $t$?  In other
words, when are the geometric analysis still work on the deforming $D(t)$,
while topological type of $D(t)$ varies in some way?
\end{prob}

This is the main difficulty of our work. We find a ``detour" strategy by
introducing the notion of \emph{set-continuity} for $D(t)$ (see
Definition~\ref{D2.7}) to replace the Hausdorff continuity, and prove
successfully the continuity of Sobolev space $H_{t}$ (see Theorem~S in
\S\ref{S4}) along the detour routes, as well as the required eigenvalue
continuity. Based on this detour strategy, the Morse index theorem is extended
to global. For the example previously mentioned, a disk $D(t)$ on a given
cylinder can thus be deformed to a ring along certain route, maintaining the
required continuities (see Figure~\ref{F6}).\\

The global Morse index theorem (i.e., Theorem~A) which answers
Problem~\ref{Prob2} is now applicable to answer Problem~\ref{Prob1} to detect
the existence of nontrivial ``Jacobi fields", which satisfy zero boundary
condition with volume constraint on generalized Lipschitz domains $D(t)$ of
various topological types. Their multiplicity can also be counted along
$t$-axis (see Theorem~J in \S\ref{S5}).

In order to establish the eigenvalue continuity, one may see that the
continuity requirement for the deforming domains $D(t)$ should be imposed on
the continuity of the boundary of $D(t)$. But then the language and the
technique would become too complicated to deal with. That is why the detour
technique chooses set-continuity deformation, which turns out to be more neat
and successful.\\

More precisely, consider the Sobolev space $H_{t}$ of variation functions on
$D(t)$, with or without volume constraint. If the continuum of $D(t)$ satisfies
the set-continuity, then $H_{t}$ is continuous in $t$. Namely, the Sobolev
continuity is established (see Theorem~S). Using the Sobolev continuity, the
required eigenvalue continuity can be shown in Theorem~B, on behalf of the
arguments in \cite{FT90} (see also \S\ref{S4}). Thus, Problem~\ref{Prob2} is
solved, and the global Morse index theorem is established. In this light, we
obtain a structural distribution theorem (see Theorem~J) of Jacobi fields on
$D(t)$, which discretely appear along the $t$-axis.\\

As an illustration, given a set-continuous continuum $\mathcal{D}$ starting
with $D(0)$, a very small neighborhood of a point in $M^{n}$. Let $t = t_{1}$
such that $D(t_{1}) \equiv D[\lambda_{1} = 0]$ is an extremal domain. There
exists a first nontrivial Jacobi field on $D(c)$, for some $c > t_{1}$. It is
clear that
\[
  D[\lambda_{1} = 0] \subset D(c) \subset D[\lambda_{2} = 0].
\]
Remark that $\partial D(c)$ is the so-called first conjugate boundary. We prove
in Theorem~J that there exists a nontrivial Jacobi field on some $D(t)$ such
that
\[ %\tag{0.1} \label{e0.1}
  D[\lambda_{k} = 0] \subset D(t) \subset D[\lambda_{k+1} = 0],
\]
where $\lambda_{k+1}$ is the adjacent eigenvalue of $L$ next to $\lambda_{k}$,
and the independent number of such Jacobi fields on $D(t)$ is at least $m_{k} +
m_{k+1} - 1 > 0$, where $m_{k}$ denotes the multiplicity of the Dirichlet
eigenvalue $\lambda_{k}$. This answers Problem~\ref{Prob1}.
\\

Technically, the Sobolev continuity, i.e., Theorem~S, is the core and difficult
part of the paper, which inevitably involves subtle arguments in the Sobolev
theory. For example, the Step~3 of the proof of Theorem~S is technically
sophisticated. We elaborated the intuitive analysis of the three closure
theorem, i.e., Theorem~Z in \S\ref{S5}, about the infinitesimal orders of
smooth functions vanishing at the corners of $D(t)$. We need the three closure
theorem on \emph{generalized} Lipschitz domains to support the Sobolev
continuity.\\

As the topological type of $D(t)$ changes, when joint points of $\partial
D(t)$(see Definition~\ref{D2.2}) appear, we have to focus on treating the
critical behavior of the relevant variation functions around the joint points.
The treatments produce the dramatic aspect of this paper.\\

It is worthwhile to note a result by Barbosa and B\'{e}rard \cite{BB00}: Given
a domain $D$ in a CMC hypersurface of $\mathbb{R}^{n+1}$, they proved that the
``twisted" eigenvalues $\{\widetilde{\lambda}_{k}\}$ of the twisted operator
$\widetilde{L}$ (see \eqref{e1.6}) and the Dirichlet eigenvalues
$\{\lambda_{k}\}$ of the stability operator $L$ are intertwined in the form
\[ %\tag{0.2} \label{e0.2}
  \lambda_{k} \leq \widetilde{\lambda}_{k} \leq \lambda_{k+1},
    \quad k = 1,2,3,\ldots.
\]
Their static result provides a partial indication of our Theorem~J. In fact, it
could be obtained by the detour approach that we develop here.

\bigskip\noindent\emph{Acknowledgement.} \quad A deep gratitude to Chun-Chi
Lin, my previous coauthor in \cite{HL98}, for the invaluable suggestions and
comments in working on this paper. Among other things, he drew my attention to
the works of [F-T] and [B-B].

%%========================================================= sec. 1 ==
\section{Preliminary} \label{S1}
%%===================================================================
We shall start with the $C^{\infty}$-framework in \S\ref{S1} and extend those
related concepts to Sobolev spaces after \S\ref{S2}. Let $M^{n}$ be a
hypersurface of constant mean curvature $H_{0}$ immersed in $\mathbb{R}^{n+1}$,
$n = \dim M^{n}$. We call such $M^{n}$ a ``CMC" hypersurface. Let $D$ be an
open set in $M^{n}$ with compact smooth boundary, which is called a smooth
domain in $M^{n}$. Define
\begin{equation} \label{e1.1}
\begin{split}
  \mathcal{F}(D)
  &:= \{ f \in C^{\infty}(D) \cap C^{1}(\overline{D}) \:;\: f|_{\partial D} = 0
    \}, \\
  \mathcal{G}(D)
  &:= \left\{ f \in \mathcal{F}(D) \:;\: \int_{D} f \, dM = 0 \right\},
\end{split}
\end{equation}
where $\mathcal{G} = \mathcal{G}(D)$ is the totality of variation functions in
$\mathcal{F}(D)$ having the volume constraint. Consider in $\mathcal{F}(D)$ the
$L^{2}$-metric
\[
  \langle f, g \rangle = \int_{D} fg \, dM.
\]
Note that the constraint $\mathcal{G}$ is a hyperplane in $\mathcal{F}(D)$. Let
$L$ be the \emph{stability operator} given by
\begin{equation} \label{e1.2}
  Lf = -\Delta_{M} f - |B|^{2} f, \quad \forall\, f \in \mathcal{F}(D),
\end{equation}
where $\Delta_{M} \equiv (*)_{ii}$, using the language of orthonormal moving
frames, is the Laplacian of $f$ relative to the metric of $M^{n}$. Note that
$|B|^{2}$ denotes $\sum_{i,j}^{n} h^{2}_{ij}$, with $h_{ij}$ the coefficients
of the second fundamental form $B$ of $M^{n}$ in $\mathbb{R}^{n+1}$.

Consider a variation field $\varphi$ on $D$, with $\varphi \in \mathcal{F}(D)$
corresponding to a $C^{\infty}$-variation $\mathbf{x}(p,\tau)$ of $D$ in
$\mathbb{R}^{n+1}$ so that $\mathbf{x}(p,0) = p \in D$, $\tau \in
(-\varepsilon,\varepsilon)$ and $\frac{\partial \mathbf{x}}{\partial \tau}(p,0)
= \varphi \cdot N$. Here $N$ is a unit normal on $D$ in $\mathbb{R}^{n+1}$.
Denote $D_{\tau} \equiv \{ \mathbf{x}(p,\tau) \:;\: p \in D \}$ and define the
functional
\[ %\tag{1.3} \label{e1.3}
  J = J(\tau) := A + n \: H_{0} V,
\]
where $A$ is the area of $D_{\tau}$, and $V$ is the volume given by
\[ %\tag{1.4} \label{e1.4}
  V
  = V(\tau)
  := \frac{1}{n+1} \int_{D_{\tau}} \langle \mathbf{x}, N \rangle \, dD_{\tau},
\]
which represents the oriented volume of a cone centered at the origin $0 \in
\mathbb{R}^{n+1}$ with base $D_{\tau}$. The first variation $J'(0)$ is zero for
any $\varphi \in \mathcal{F}(D)$. This is a character of CMC hypersurfaces with
mean curvature $H_{0}$.

The second variation of $J$ relative to $\varphi \in \mathcal{G}(D)$ is
\[ %\tag{1.5} \label{e1.5}
  J''(0) = \int_{D} L\varphi \cdot \varphi \, dM.
\]
When $J''(0) \geq 0$, for all $\varphi \in \mathcal{G}(D)$, we say that $D$ is
\emph{stable}. There exist eigenfunctions of $L$ on $\mathcal{F}(D)$,
\[
  u_{1}, u_{2}, u_{3}, \ldots,
\]
which constitute an orthonormal basis of $\mathcal{F}(D) \subset L^{2}(D)$ and
satisfy
\begin{equation} \label{e1.3}
  Lu_{k} = \lambda_{k} \: u_{k}, \quad k = 1,2,3,\ldots,
\end{equation}
with eigenvalues $\lambda_{k}$ in the form,
\[ %\tag{1.7} \label{e1.7}
  \lambda_{1} < \lambda_{2} \leq \lambda_{3} \leq \lambda_4 \leq \cdots
  \to \infty.
\]
Note that
\[ %\tag{1.8} \label{e1.8}
  D' \subset D'' \;\; \textrm{implies} \;\;
  \lambda_{k}(D') \geq \lambda_{k}(D'')
\]
for each $k = 1,2,3,\ldots$.

\begin{defn} \label{D1.1}
A domain $D$ in $M^{n}$ is called \emph{extremal} if
\[
  \lambda_{1}(D) = 0.
\]
\end{defn}

An extremal domain $D$ has interesting geometry properties such as the
variation rate $\frac{dH_{\tau}}{dV_{\tau}}$ of $D_{\tau}$ changes sign on $D$,
where the variation function is in $\mathcal{F}(D)$. Here $V_\tau$ is the
volume parameter, and $H_{\tau}$ denotes the mean curvature of $D_{\tau}$ in
$\mathbb{R}^{n+1}$ such that $H_{\tau}$ is constant on $D_{\tau}$. For example,
see this on a hemisphere.

It also provides a criterion for the extent of concavity of CMC surfaces. The
negative curved set $M_{-}$ of a CMC surface $M$ must be \emph{large} in the
sense that $M_{-}$ contains at least an extremal domain, unless $M_{-}$ touches
the given boundary of $M$.  Note that any proper subdomain of a nonparametric
$M$ is not an extremal domain. It becomes extremal if and only if its domain of
definition is pushed to the extreme that M can not be extended any further (see
\cite{HL98}). We will show more significant properties of extremal domains
related to stability and Jacobi fields, later in the paper.

Define the bilinear form $I = I_{D} \colon \mathcal{F}(D) \times \mathcal{F}(D)
\to \mathbb{R}$ by
\begin{align}
\label{e1.4}
  I(f,g)
  &:=  \int_{D} (Df \cdot Dg - |B|^{2} fg) \, dM \\
\label{e1.5}
  &= \int_{D} Lf \cdot g \, dM.
\end{align}
Remark that
\[ %\tag{1.9c} \label{e1.9c}
  J''(0) = \int_{D} Lf \cdot f \,dM = I(f,f),
\]
$\forall\, f \in \mathcal{G}(D)$. We also write $I(f) \equiv I(f,f)$ for
simplicity.

\begin{defn}[Unstable cones] \label{D1.2}
Given a domain $D$ in $M^{n}$, the \emph{unstable cones} in the ``Dirichlet
sense" (i.e., without volume constraint) are defined by
\begin{align*}
%\tag{1.10a} \label{e1.10a}
  \Lambda &:= \{ f \in \mathcal{F}(D); I(f) \leq 0 \}, \\
%\tag{1.10b} \label{e1.10b}
  \Lambda_{-} &:= \{ f \in \mathcal{F}(D); f = 0 \textrm{ or } I(f) < 0 \}.
\end{align*}
Call $\widetilde{\Lambda} := \Lambda \cap \mathcal{G}$ and
$\widetilde{\Lambda}_{-} := {\Lambda}_{-} \cap \mathcal{G}$ the \emph{unstable
cones} in the ``constraint" $\mathcal{G}$. Note that $D$ is stable if and only
if $\widetilde{\Lambda}_{-} = \{0\}$; and $D$ is unstable if and only if
$\widetilde{\Lambda}_{-} \neq \{0\}$.
\end{defn}

\begin{defn}[Jacobi fields] \label{D1.3}
Let $N$ be the unit normal vector field of $M$ and $\varphi \in \mathcal{G} =
\mathcal{G}(D)$ on $D$. A variation field $\varphi \cdot N$ is called a
(nontrivial) \emph{Jacobi field} on $D$, if $\varphi \neq 0$ (i.e., $\varphi$
is not identically zero), and $I(\varphi,g) = 0$, $\forall\, g \in
\mathcal{G}$, i.e., $\varphi \in \operatorname{Ker} I$ on $\mathcal{G}$. In
fact, this definition is equivalent to saying that $L\varphi = 0$ in the
classical sense. Note that $\varphi$ should be called a Jacobi variation
function. But it is conventional to call it just a ``Jacobi field".
\end{defn}

\begin{defn}[Conjugate boundaries] \label{D1.4}
A domain on which a Jacobi field exists  (in the sense given by
Definition~\ref{D1.3}) is called having \emph{conjugate boundary}. Moreover, if
a domain has a conjugate boundary, but has no such subdomain smaller, the
domain is called having the \emph{first conjugate boundary}.
\end{defn}

\begin{rmk} \label{R1.5}
There are two notions of stability for domains on CMC hypersurfaces. The
extremal domains are related to the notion of the Dirichlet stability, while
domains with the first conjugate boundary are related to the notion of the
stability with volume constraint.
\end{rmk}

\begin{defn} \label{D1.6}
We say that the cone $\Lambda$ \emph{tangents to} the hyperplane $\mathcal{G}$
at $\varphi \in \mathcal{G}$, if $\varphi \neq 0$, $\varphi \in
\widetilde{\Lambda}$, and $\widetilde{\Lambda}_{-} = \{0\}$ (see
Figure~\ref{F1}).
\begin{figure}[H]
\centering
\includegraphics[width=0.7\textwidth]{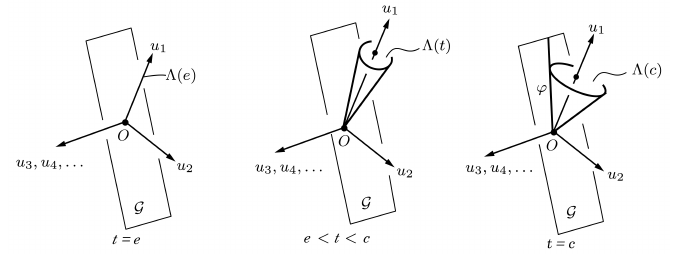}
\caption{} \label{F1}
\end{figure}
\end{defn}

\begin{prop} \label{P1.7}
Given $D \subset M^{n}$, let the cone $\Lambda$ tangent to $\mathcal{G}$ at
$\varphi \in \mathcal{G}$, then $\varphi$ is a Jacobi field on $D$. \end{prop}
\begin{proof}
Clearly, $I(f) \geq 0$, for any $f \in \mathcal{G}$, since
$\widetilde{\Lambda}_{-} = \{0\}$. Given $g \in \mathcal{G}$, we claim that
$I(\varphi,g) = 0$. In fact, define $F(t) = I(\varphi+tg)$, then $F(t) \geq 0$
for each real number $t$, since $\varphi+tg \in \mathcal{G}$. Also, $F(0) =
I(\varphi) = 0$ as $\varphi \in \widetilde{\Lambda} - \widetilde{\Lambda}_{-}$.
Thus $F(t)$ attains minimum at $t = 0$. We have
\begin{align*}
  0
  &= F'(0)
  = \frac{d}{dt} \big( I(\varphi+tg,\varphi+tg) \big) \big|_{t=0} \\
  &= \frac{d}{dt} \big( I(\varphi,\varphi) + 2tI(\varphi,g)
    + t^{2} I(g,g) \big) \big|_{t=0}
  = 2I(\varphi,g). \qedhere
\end{align*}
\end{proof}

\begin{prop} \label{P1.8}
Given $D \subset M^{n}$, $\varphi \in \mathcal{G}$ is a Jacobi field on $D$, if
and only if
\[ %\tag{1.11} \label{e1.11}
  L\varphi = \frac{1}{|D|} \int_{D} L\varphi \, dM \in \mathbb{R},
\]
where $|D| := \operatorname{area}(D)$.
\end{prop}
\begin{proof}
When $L\varphi$ = a constant $c$,
\[
  I(\varphi,g)
  = \int_{D} L\varphi \cdot g \, dM
  = c\int_{D} g \, dM
  = 0, \quad \forall\, g \in \mathcal{G}.
\]
Hence $\varphi$ is a Jacobi field on $D$. By integration on $D$, $c$ must
satisfy $c = \frac{1}{|D|} \int_{D} L\varphi \, dM$. Conversely, suppose
$\varphi$ is a Jacobi field on $D$ and $L\varphi$ is not a constant, then there
exist two points $p$ and $q$ in $D$ such that $L\varphi(p) \neq L\varphi(q)$.
Construct $g \in \mathcal{G}$ by letting $g = 0$ on $D - N_{p} \cup N_{q}$,
where $N_{p}$ and $N_{q}$ are respectively small and disjoint neighborhoods of
$p$ and $q$ such that $|N_{p}| = |N_{q}|$. Now define $g \in \mathcal{G}$ by
letting ``basically" $g = +1$ in $N_{p}$ and $g = -1$ in $N_{q}$. One may
smooth out $g$ so that $g \in \mathcal{G}$ and $\int_{D} L\varphi \cdot g \, dM
\neq 0$. This contradicts the assumption that $\varphi \in \mathcal{G}$ is a
Jacobi field.
\end{proof}

Let $\widetilde{L} \colon \mathcal{G} \to C^{\infty}(D)$ be defined by
\begin{equation} \label{e1.6}
  \widetilde{L}g
  := Lg - \frac{1}{|D|} \int_{D} Lg, \quad \forall\, g \in \mathcal{G},
\end{equation}
omitting ``$dM$" hereafter, if no ambiguity arises. Clearly, $\int_{D}
\widetilde{L}g = 0$, and $\varphi \in \mathcal{G}$ is a Jacobi field on $D$ if
and only if $\widetilde{L}\varphi = 0$.

\begin{thm}[Jacobi fields and unstability] \label{T1.9}
Suppose that $D \subset M^{n}$ has conjugate boundary, i.e., there exists a
Jacobi field $\varphi$ on $D$. Then, $D'$ is unstable, $\forall\, D' \subset
M^{n}$ with $D'\supsetneqq D$.
\end{thm}

\begin{rmk} \label{R1.10}
Theorem~\ref{T1.9} relates to the initial case of the Morse index theorem,
which establishes the correspondence of nullity and index (see \eqref{e2.17} in
Theorem~A). The existence of one (independent) Jacobi field represents ``one
nullity", while the appearance of one (independent) nonzero unstable variation
means ``one index".
\end{rmk}

\begin{proof}
\emph{Step~$1$.} Consider $g = \frac{1}{\alpha} \psi + \alpha h \in
\mathcal{G}(D')$, $\alpha > 0$ very small, where
\[ %\tag{1.13} \label{e1.13}
  \psi = \begin{cases}
  \varphi &\textrm{on $D$}, \\
  0 &\textrm{on $G = D'-D$},
  \end{cases}
\]
\begin{figure}[h]
\centering
\includegraphics[width=0.25\textwidth]{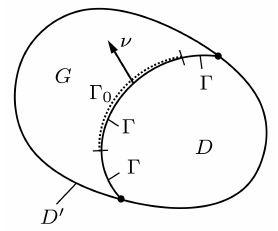}
\caption{} \label{F2}
\end{figure}
and $h \in \mathcal{G}(D')$ such that (see Figure~\ref{F2})
\[ %\tag{1.14} \label{e1.14}
  h|_{\Gamma} = -\beta \frac{\partial \varphi}{\partial \nu}, \quad
  \int_{D'} h = 0, \quad \Gamma := \partial D \cap D',
\]
where $\beta \in \mathcal{F}(D')$ with $\beta \geq 0$ and $\beta|_{\Gamma_{0}}
= 1$, $\Gamma_{0} \subset\subset \Gamma$.

\bigskip\noindent\emph{Step~$2$.} We shall use the method of moving frames to
simplify the following computation. Write for abbreviation that
\[ %\tag{1.15} \label{e1.15}
  I_{D'}(u,v) = \int_{D'} (u_i v_i - |B|^{2} uv),
\]
where $u_i = D_i u$. Also, the summation notation and the area element ``$dM$"
are omitted. Consider the second variation $J''(0)$ along $g$ on $D'$. Note
that $g$ is piecewise smooth on $D'$ and \eqref{e1.5} is not well defined, but
formula~\eqref{e1.4} is still valid. We have
\begin{gather*}
\begin{split} %\tag{1.16} \label{e1.16}
  J''_{D'}(0)
  &= I_{D'}(g)
  = I_{D'} \left( \frac{1}{\alpha} \psi + \alpha h, \:
    \frac{1}{\alpha} \psi + \alpha h \right) \\
  &= \frac{1}{\alpha^{2}} I_{D}(\varphi,\varphi) + 2I(\varphi,h)
    + \alpha^{2} I_{D'}(h,h) \\
  &= \mathrm{I} + \mathrm{II} + \mathrm{III}, \quad \textrm{respectively},
\end{split} \\
\begin{split} %\tag{1.17} \label{e1.17}
  \mathrm{I}
  &= \frac{1}{\alpha^{2}} \int_{D} (\varphi_i \varphi_i - |B|^{2} \varphi^{2})
  = \frac{1}{\alpha^{2}} \int_{D}
    \big( (\varphi_i \varphi)_i - \varphi_{ii} \varphi - |B|^{2} \varphi^{2}
    \big) \\
  &= \frac{1}{\alpha^{2}} \left( \int_{\partial D}
    \frac{\partial \varphi}{\partial \nu} \varphi
    + \int_{D} L\varphi \cdot \varphi \right)
  = 0,
    \quad \textrm{(since $\varphi \big|_{\partial D} = 0$ and $\varphi$ is
    Jacobi on $D$)}
\end{split} \\
\begin{split} %\tag{1.18} \label{e1.18}
  \mathrm{II}
  &= 2\int_{D} (\varphi_i h_i - |B|^{2} \varphi h)
  = 2\int_{\partial D} \frac{\partial \varphi}{\partial \nu} \cdot h
    + 2\int_{D} L\varphi \cdot h \\
  &= -2 \int_{\Gamma} \beta \left( \frac{\partial \varphi}{\partial \nu}
    \right)^{2} + 0, \quad \textrm{(since $\varphi$ is Jacobi on $D$)}.
\end{split}
\end{gather*}
Hence
\[ %\tag{1.19} \label{e1.19}
  J''(0) = 0 + (-A) + \alpha^{2} I_{D'}(h,h),
\]
where $-A$ is a negative number independent of $\alpha$, since $\frac{\partial
\varphi}{\partial \nu} \neq 0$ by Hopf's sphere theorem. Choosing $\alpha$
sufficiently small, we have $J''_{D'}(0) < 0$, i.e., $D'$ is unstable.
\end{proof}

\begin{rmk} \label{R1.11}
For precision’s sake, let $\gamma$ be an immersion of $M^{n}$ into
$\mathbb{R}^{n+1}$ as a hypersurface of constant mean curvature. Consider the
metric on $M^{n}$ induced by $\gamma$ from $\mathbb{R}^{n+1}$. We may employ
the calculus of variations on a domain $D$ in $M^{n}$, with variation functions
in $\mathcal{F}(D)$ or in $\mathcal{G}(D)$ defined by \eqref{e1.1}, where the
stability operators $L$ or $\widetilde{L}$ is given by \eqref{e1.2} or
\eqref{e1.6}, respectively. The setting can be handled purely by analysis. When
we like to consider $D$ in a CMC hypersurface $N^{n}$ sitting in
$\mathbb{R}^{n+1}$, we merely regard $\varphi$ as the identity map of $N^{n}$,
letting $M^{n}$ = $N^{n}$. This version applies to the later sections where
continuums of domains $D(t)$ in $M^{n}$ are considered.
\end{rmk}

%%========================================================= sec. 2 ==
\section{Sobolev variation and a global Morse index theorem} \label{S2}
%%===================================================================
We shall extend the previous smooth setting to consider Sobolev spaces of
variation functions on generalized Lipschitz domains (see
Definitions~\ref{D2.1} and \ref{D2.2}). The advantage of treating the
``generalized" notion of Lipschitz domains is that the topological type of $D =
D(t)$ will be allowed to change in $t$.

%%======================================================= sec. 2.1 ==
\subsection{The Sobolev setting} \label{S2.1}
%%===================================================================
Let $W^{1,2}(D)$ be the Sobolev space consisting of all $f$ in $L^{2}(D)$,
which have weak derivatives $Df$ with $Df \in L^{2}(D)$. Consider
\begin{align*}
  \mathcal{F}_{m}(D)
  &:= \{ f \in C^{\infty}(D) \cap C^{m}(\overline{D}) \:;\: f|_{\partial D} = 0
    \}, \\
  \mathcal {G}_{m}(D)
  &:= \left\{ f \in F_{m}(D) \:;\: \int_{D} f \, dM = 0 \right\}
\end{align*}
with $m$= 0 or 1, and define two spaces $E(D)$ and $H(D)$ of variation
functions on $D$ in $W^{1,2}(D)$ by
\begin{equation} \label{e2.1}
\begin{split}
  E(D)
  &:= \textrm{ the closure of $\mathcal{F}_{0}(D) \cap W^{1,2}(D)$ in
    $W^{1,2}(D)$}, \\
  H(D)
  &:= \textrm{ the closure of $\mathcal{G}_{0}(D) \cap W^{1,2}(D)$ in
    $W^{1,2}(D)$},
\end{split}
\end{equation} \\
Remark that $\mathcal{F}_{1}(D)=\mathcal{F}(D)$ and
$\mathcal{G}_{1}(D)=\mathcal{G}(D)$, where $\mathcal{F}(D)$ and
$\mathcal{G}(D)$ are given in \eqref{e1.1}. By the three closure theorem proved
in \S3.3, we will see that if $\mathcal{F}_{0}(D)$ is replaced by
$\mathcal{F}_{1}(D)$ in \eqref{e2.1}, $E(D)$
keeps the same. Similar observation applies to $H(D)$.
\\

The extension of variation functions on $D$, from smooth setting to $E(D)$ and
$H(D)$ in the Sobolev framework, permits us to include, for example, piecewise
smooth functions as variation functions (see $g$ and $\phi$ in the proof of
Theorem~\ref{T1.9}). Another obvious benefit is that $E(D)$ and $H(D)$ are
Hilbert spaces in which completeness is guaranteed. Furthermore, for a monotone
family $\mathcal{D} \equiv \{ D(t) \:;\: 0 < t \leq b \}$ of domains in $M^{n}$
(see Definition~\ref{D2.7}), we regard $f \in H(D(t)) \subset H(D(b))$,
$\forall\, t \leq b$, by letting $f \equiv 0$ outside $D(t)$. In this sense,
all the variation functions on $D(t)$ can be treated in the same ambient space
$H(D(b))$. Thus, the unstable cone $\widetilde{\Lambda} =
\widetilde{\Lambda}(D(t))$ enlarging as $t$ increases will be well-defined.
Also, a Jacobi field on $D(t)$ (see Definition~\ref{D1.3}) can be included as
an element in $H(D(r))$ for any $r > t$. In the sequel, we shall freely use
simplified notations such as $\mathcal{G} \equiv \mathcal{G}(D)$, $H \equiv
H(D)$, $H_{t} \equiv H(D(t))$, $H_{t} \subset H_{b}$, etc., just for
convenience.

\begin{defn} \label{D2.1}
A \emph{simple domain} $U$ in $\mathbb{R}^{n}$ is an open set in $\mathbb{R}^{n}$,
diffeomorphic to an $n$-ball. A \emph{simple triple} $(U,\Gamma,V)$ is a
simple domain $U$, attached with $\Gamma \subset \partial U$, given by
\[ %\tag{2.b} \label{e2.b}
  \Gamma = \{ (x,u(x)) \in \mathbb{R}^{n} \:;\: x \in V \},
\]
and an open set $V$ in a hyperplane $H^{n-1}$ of $\mathbb{R}^{n}$, where $V$ is
diffeomorphic to an $(n-1)$-ball. When $u = u(x)$ is a Lipschitz function of $x
\in V$, we say the triple a \emph{Lipschitz simple domain} (see Figure~\ref{F3}).
\end{defn}

We will give a precise definition of a generalized Lipschitz domain by gluing together some boundary points of a Lipschitz domain.

\begin{defn} \label{D2.2}
Let $D'$ be a compact $C^\infty$-manifold of dimension $n$ with $C^\infty$-boundary $\partial{D'}$, and let $\rho : D'\rightarrow M^{n}$ be a $C^\infty$-map, such that the restriction $\rho|_{Int D'}$ is a $C^\infty$-diffeomorphism, letting Int$A$ denote the interior of a set $A$ in $M^n$. We call $\rho$ a \emph{boundary gluing map} of $D'$, if $\rho$ maps at most finite number of points of $\partial{D'}$ to any given point in $M^n$. We may find a family \{$(U'_i,{\Gamma}'_i,V'_i);i=1,\cdots,\nu$\} of simple triples \emph{covering $\partial D'$ in the sense} that (i) each $U'_i$ is open in $D'$ with $U'_1\cup\cdots\cup U'_\nu \supset \partial D'$; (ii) each $\Gamma'_i$ is an open set of $\partial D'$, $\Gamma'_{i} \subset \partial U'_i$ and $\Gamma'_1 \cup \cdots,\cup \Gamma'_\nu$= $\partial D'$; (iii) each $(U'_i,\gamma'_i,V'_i)$ is a simple triple lying in $\mathbb{R}^n$(see Definition~\ref{D2.1}), through a coordinate map of $D'$. Given $D$ a domain in $M^n$, we say that $D$ is a \emph{generalized Lipschitz domain} in $M^n$, if there exists a boundary gluing map $\rho :D' \rightarrow M^n$ with $D= \rho(Int D')$ and a family of simple triples \{$(U'_i,\Gamma'_i, V'_i); i=1,\cdots,\nu$\} covering $\partial D'$, such that $(U_i,\Gamma_i,V_i)$ is a Lipschitz simple domain in $M_n$ again through a coordinate map of $M^n$, where $U_i,\Gamma_i$ and $V_i$ are the images of $U'_i, \Gamma'_i$ and $V'_i$ under the map $\rho$ respectively.

\begin{figure}[H]
\begin{minipage}[t]{0.35\textwidth}
\centering
\includegraphics{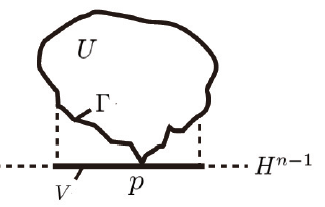}
\caption{simple Lipschitz domain} \label{F3}
\end{minipage} \quad
\begin{minipage}[t]{0.6\textwidth}
\centering
\includegraphics{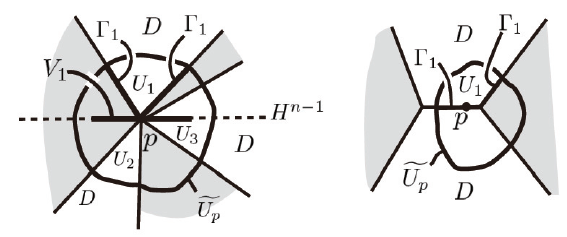}
\caption{generalized Lipschitz domain (local version)} \label{F4}
\end{minipage}
\end{figure}
\end{defn}

A point $p$ in $M^n$ is called a \emph{joint point of} $\rho$, or by abuse of language, a \emph{joint point of} $\partial D$, if the preimage $\rho^{-1}(p)$ contains more than one element of $\partial D'$, i.e., if some different points of $\partial D'$ are glued to the point $p \in \partial D$ through $\rho$. Remark that a joint point of $\partial D$ may not be in $\partial D$. (See the right diagram in Figure~\ref{F4} ). The \emph{joint set of} $\partial D$ is the totality of all joint points of $\partial D$. Correspondingly, we also call a point $p'\in \partial D'$ a \emph{joint point of} $\partial D'$, if $\rho$ is not injective at $p'$. The \emph{joint set of} $\partial D'$ means the set of all joint points of $\partial D'$.\\

Hereafter, the domains such as $D$ or $D(t)$ in $M^{n}$ stand for generalized
Lipschitz domains. Sometimes we \emph{still} call them Lipschitz domains for convenience.

\begin{defn} \label{D2.4}
Let $X$ be a metric space, $K \subset X$ a relatively compact set, i.e., the
closure $\overline{K}$ being compact. Given $\alpha \geq 0$, define the
$\alpha$-tubular neighborhood of $K$ by
\[ %\tag{2.1} \label{e2.1}
  K^{\alpha} := \{ x \in X \:;\: d(x,K) \leq \alpha \}
\]
and the Hausdorff distance $d^{H}$ between two relatively compact sets $K$ and
$L$ by
\begin{equation} \label{e2.2}
  d^{H}(K,L)
  := \inf \{ \alpha \geq 0 \:;\: K \subset L^{\alpha} \textrm{ and } L \subset
    K^{\alpha} \}.
\end{equation}
\end{defn}

\begin{defn} \label{D2.5}
Given a family $\mathcal{D} := \{ D(t) \subset M^{n} \:;\: t \in [0,b] \}$,
where each $D(t)$ is a (generalized) Lipschitz domain in $M^{n}$. We say that
$\mathcal{D}$ is a \emph{monotone family} of Lipschitz domains, if
\[ %\tag{2.3} \label{e2.3}
  s < t \;\; \Rightarrow \;\; D(s) \subsetneqq D(t),
    \quad \forall\, s,t \in [0,b].
\]
\end{defn}

\begin{defn} \label{D2.6}
Given a monotone family $\mathcal{D}$. If $\forall\, \epsilon > 0$, there
exists $\delta > 0$ such that
\begin{equation} \label{e2.3}
  |r-t| < \delta \;\; \Rightarrow \;\; d^{H}(D(r),D(t)) < \epsilon,
\end{equation}
$\forall\, r,t \in [0,b]$, we say that $\mathcal{D}$ is \emph{continuous} in
$d^{H}$, or \emph{Hausdorff continuous}.
\end{defn}

Let $\overline{D}$ denote the closure of $D$ in $M^{n}$.

\begin{defn} \label{D2.7}
Given a monotone family $\mathcal{D}$. If
\begin{equation} \label{e2.4}
  D(t) = \bigcup_{s < t} D(s), \quad
  \overline{D(t)} = \bigcap_{r > t} \overline{D(r)},
\end{equation}
and furthermore, for any $t \in [0,b]$, we have
\begin{equation} \label{e2.5}
  \partial(\overline{D(t)}) = \partial(D(t)),
\end{equation}
then we say that $\mathcal{D}$ is \emph{set-continuous} in $t$, and call such
$\mathcal{D}$ a \emph{$C^{0}$-monotone continuum} of domains in $M^{n}$.
\end{defn}

\begin{rmk} \label{R2.8}
In a $C^{0}$-monotone continuum of domains in $M^{n}$, the joint set of
$\partial{D(t)}$ defined in Definition~\ref{D2.2} contains no open set of some
$\Gamma_{j}$, such as in the right diagram of Figure~\ref{F4}, since otherwise \eqref{e2.5} would be violated. However, at an isolated joint point, the topological type of $D(t)$ can be changed along $t$. Homotopically, it is equivalent to adding a tube around the isolated joint point (see Figure~\ref{F6} and Example~\ref{E3.2} in \S\ref{S3.1}). Thus, the topological type of $D(t)$ may change in various ways.
\end{rmk}

In this paper, we consider $\mathcal{D}$ starting with a small neighborhood
$D(0)$ of a given point $p_{0} \in M^{n}$, although essentially it is not
necessary.

The set-continuity implies the Hausdorff continuity in $d^{H}$, but not vice
versa. The proof will be included later as a special case of Step~1 in the
proof of Theorem~S (see \S\ref{S3}). Examples~\ref{E3.3} and \ref{E3.4} given
in \S\ref{S3.2} compare the two notions of Hausdorff continuity and
set-continuity. Example~\ref{E3.2} in Figure~\ref{F6} of \S\ref{S3.1} also
shows a significant example of $C^{0}$-monotone continuum.\\

Let $D \subset M^{n}$ be a generalized Lipschitz domain, and $W^{k,2}(D)$
denote the Sobolev space of functions with weak derivatives up to the $k$-th
order in $L^{2}(D)$. We write
\begin{equation} \label{e2.6}
\begin{split}
  E^{k}
  &\equiv E^{k}(D)
  := \textrm{ the closure of $\mathcal{F}_{0}(D) \cap W^{k,2}(D)$ in
  $W^{k,2}(D)$},
    \\
  H^{k}
  &\equiv H^{k}(D)
  := \textrm{ the closure of $\mathcal{G}_{0}(D) \cap W^{k,2}(D)$ in
  $W^{k,2}(D)$}.
\end{split}
\end{equation}
For $k = 1$, recall that we have written $E \equiv E^{1}$, $H \equiv H^{1}$ in
\eqref{e2.1} for simplicity, as the two spaces are the ones of our main
concern. Note that $E^{0} = L^{2}(D)$, since the boundary condition
$f\big|_{\partial D} = 0$ has no restriction in $L^{2}(D)$.\\

Extend the stability operator $L$ defined by \eqref{e1.2} on $\mathcal{F}(D)$
to $E^{2}(D)$, i.e., consider $L \colon E^{2} \to E^{0}$ by
\begin{equation} \label{e2.7}
  Lf := -\Delta_{M} f - |B|^{2} f, \quad \forall\, f \in E^{2}.
\end{equation}
Similarly, extend the bilinear form $I$ of \eqref{e1.4} to $E \equiv E^{1}(D)$
by
\begin{equation} \label{e2.8}
  I(f,g) = \int_{D} Df \cdot Dg - |B|^{2} fg, \quad \forall\, f,g \in E.
\end{equation}
We see that $I$ is continuous and symmetric on $E$, satisfying
\[ %\tag{2.9} \label{e2.9}
  \langle Lf, g \rangle = I(f,g), \quad \forall\, f \in E^{2}, \; g \in E,
\]
and $\langle Lf, g \rangle = \langle f, Lg \rangle$ on $E^{2}$. Clearly, $H
\hookrightarrow H^{0}$ is a compact embedding. Given a domain $D$ in $M^{n}$,
we define $\widetilde{L} \colon H^{2} \to H^{0}$ by letting
\begin{equation} \label{e2.9}
  \widetilde{L}f
  := -\Delta_{M} f - |B|^{2} f - \frac{1}{|D|} \int_{D} Lf,
    \quad \forall\, f \in H^{2},
\end{equation}
as in \eqref{e1.6}. Introduce correspondingly a bilinear form $\widetilde{I}$
on $H$, given by
\begin{equation} \label{e2.10}
  \widetilde{I}(f,g) = \int_{D} Df \cdot Dg - |B|^{2} fg, \quad f,g \in H.
\end{equation}
Clearly, $I$ restricted to $H \subset E$ is $\widetilde{I}$, i.e.,
\[ %\tag{2.12} \label{e2.12}
  \widetilde{I}(f,g) = I(f,g), \quad \forall\, f,g \in H.
\]
Still, $\widetilde{I}$ is continuous and symmetric on $H$, satisfying
\[ %\tag{2.13} \label{e2.13}
  \langle \widetilde{L}f, g \rangle = \widetilde{I}(f,g),
    \quad \forall\, f \in H^{2}, \; g \in H
\]
and $\langle \widetilde{L}f,g \rangle = \langle f, \widetilde{L}g \rangle$,
$\forall\, f,g \in H^{2}$, i.e., $\overline{L}$ is self-adjoint on $H^2$.

%%======================================================= sec. 2.2 ==
\subsection{The spectral theorem} \label{S2.2}
%%===================================================================
We need to check the spectral theorem of $\widetilde{L}$ on $H$ for
completeness. Let $\{\widetilde{u}_{k}\}$ be the eigenfunctions of
$\widetilde{L}$, which constitute an orthonormal basis of $H$ with
\begin{gather}
\label{e2.11}
  \widetilde{L} \; \widetilde{u}_{k}
  = \widetilde{\lambda}_{k} \; \widetilde{u}_{k}, \quad k = 1,2,3,\ldots, \\
\notag %\tag{2.16} \label{e2.16}
  \widetilde{\lambda}_{1}
  \leq \widetilde{\lambda}_{2}
  \leq \widetilde{\lambda}_{3}
  \leq \cdots
  \to \infty.
\end{gather}
We also obtain from the proof of the spectral theorem the min-max principle
\begin{align}
\label{e2.12}
  \widetilde{\lambda}_{k}
  &= \min \{ \widetilde{I}(f,f) \:;\: f \in S_{k} \} \\
\label{e2.13}
  &= \min_{V^{k}} \: [\: \max \{ \widetilde{I}(f,f) \:;\: f \in V^{k} \cap S \}
    \:],
\end{align}
where $S$ is the unit sphere $\{ f \in H \:;\: \langle f, f \rangle = 1 \}$ in
$H$,
\begin{equation} \label{e2.14}
  S_{k}
  := \{ f \in S \:;\: \langle f, \widetilde{u}_{j} \rangle = 0,
    \,\forall\, j = 1,2,\ldots,k-1 \},
\end{equation}
and $V^{k} \subset H$ denotes $k$-dimensional linear subspaces of $H$. Remark
that some of $\widetilde{\lambda}_{k}$ may be negative. From \eqref{e2.13} and
\eqref{e2.14}, we see that
\begin{equation} \label{e2.15}
  D \subset D' \;\; \Rightarrow \;\;
  \widetilde{\lambda}_{k}(D) \geq \widetilde{\lambda}_{k}(D'),
\end{equation}
since $V^{k} \cap S$ for $D$ is a subset of that for $D'$, and the set of
maxima in \eqref{e2.13} for $D$ constitutes in the real line a subset of that
for $D'$.

By the weak equation~\eqref{e2.11}, we mean
\[ %\tag{2.21} \label{e2.21}
  \langle \widetilde{\lambda}_{k} \widetilde{u}_{k}, g \rangle
  = \langle \widetilde{L} \widetilde{u}_{k}, g \rangle
  = \int_{D} D\widetilde{u}_{k} \cdot Dg - |B|^{2} \, \widetilde{u}_{k}g
  = \widetilde{I}(\widetilde{u}_{k},g), \quad \forall\, g \in H,
\]
noting that the term $(\frac{-1}{|D|} \int_{D} L\widetilde{u}_{k}) \int_{D} g
= 0$, as $g \in H$. As an illustration, we shall show for the operator
$\widetilde{L}$ the regularity of $\widetilde{u}_{k}$, and the expansion
formula~\eqref{e2.16}.

\begin{rmk} \label{R2.9}
Any given eigenfunction $\widetilde{u}$ of $\widetilde{L}$ is smooth, i.e.,
$\widetilde{u}_{k} \in \mathcal{G}_{0}(D) \subset C^{\infty}(D) \cap
C^{0}(\overline{D})$. If $\partial D$ is smooth, then $\widetilde{u}_{k} \in
C^{\infty}(\overline{D})$.
\end{rmk}
\begin{proof}
In general, for $u \in H \subset W^{1,2}(D)$, it is known that a weak solution
of $D_{i}(a^{ij} D_{j}u) = F$ is in $W^{k+2,2}(D)$, if $a^{ij} = a^{ij}(x) \in
C^{k+1}(D)$, $F = F(x) \in W^{k,2}(D)$, $x \in D$. Specially, given
$\widetilde{u} \in H$, a weak solution of $\widetilde{L} \widetilde{u} =
\widetilde{\lambda} \widetilde{u}$, let $F(x) := |B|^{2} + c -
\widetilde{\lambda} \widetilde{u}(x)$, where $c$ is the constant $\frac{1}{|D|}
\left( \int_{D} L\widetilde{u} \right)$, then $F \in W^{1,2}(D)$, and hence
$\widetilde{u} \in W^{3,2}(D)$. Iteratively, we see that $\widetilde{u} \in
W^{l,2}(D)$ for any $l \geq 0$. By the Sobolev embedding theorem, $W^{l,p}(D)
\subset C^{m}(D)$, with any $m$ such that $0 \leq m < l - \frac{n}{p}$. Now, $p
= 2$ and $l$ is arbitrarily large, therefore $\widetilde{u} \in C^{\infty}(D)$.
However, $\widetilde{u} \in H$, it is clear that $\widetilde{u} \in
\mathcal{G}_{0}(D) \subseteq C^{\infty}(D) \cap C^{0}(\overline{D})$. The proof
of the second statement is standard.
\end{proof}

\begin{rmk} \label{R2.10}
For any $v \in H$, we have
\begin{equation} \label{e2.16}
  v = a_{1} \widetilde{u}_{1} + a_{2} \widetilde{u}_{2} + \cdots,
    \quad a_{k} = \langle v, \widetilde{u}_{k} \rangle,
\end{equation}
where the expansion~\eqref{e2.16} is in $L^{2}(D)$.
\end{rmk}
\begin{proof}
Given $v \in H$, define $v_{k} := \langle v, \widetilde{u}_{1} \rangle
\widetilde{u}_{1} + \cdots + \langle v, \widetilde{u}_{k-1} \rangle
\widetilde{u}_{k-1}$, and $w_{k} \equiv v-v_{k}$. It suffices to show that
$\|w_{k}\|^{2} \to 0$ as $k \to \infty$. Evidently, $w_{k}/\|w_{k}\| \in
S_{k}$. By \eqref{e2.12}, $\widetilde{I}(w_{k},w_{k}) \geq
\widetilde{\lambda}_{k} \|w_{k}\|^{2}$. However,
\begin{align*} %\tag{2.23} \label{e2.23}
  \widetilde{\lambda}_{k} \|w_{k}\|^{2}
  &\leq \widetilde{I}(w_{k},w_{k})
  = \widetilde{I}(v,v) - 2\widetilde{I}(w_{k},v_{k})
    - \widetilde{I}(v_{k},v_{k}) \\
  &\leq \widetilde{I}(v,v) - \widetilde{\lambda}_{1} \|v_{k}\|^{2}
  \leq \widetilde{I}(v,v) + |\widetilde{\lambda}_{1}| \|v\|^{2},
\end{align*}
in which $\widetilde{I}(w_{k},v_{k}) = 0$.  Note that $\|v_{k}\|^{2} \leq
\|v\|^{2}$. For $k \to \infty$, $\widetilde{\lambda}_{k} \to \infty$, we see
that $\|w_{k}\|^{2} \to 0$, as required.
\end{proof}

%%======================================================= sec. 2.3 ==
\subsection{A global Morse index theorem} \label{S2.3}
%%===================================================================
\begin{defn} \label{D2.11}
Let $X$ be a Hilbert space on which a bilinear form $I \colon X \times X \to
\mathbb{R}$ is continuous and symmetric. Define $\operatorname{Ker} I$ by $\{ x
\in X \:;\: I(x,y) = 0, \,\forall\, y \in X \}$. Let the nullity $\nu$ and the
index $i$ of $I$ be given by
\[
%\tag{2.21a} \label{e2.21a}
  \nu := \dim(\operatorname{Ker} I) \quad \textrm{and} \quad
%\tag{2.22a} \label{e2.22a}
  i := \dim W,
\]
where $W$ is a maximal linear subspace of $X$ contained in the unstable cone
\[ %\tag{2.23a} \label{e2.23a}
  \Lambda_{-} := \{ x \in X \:;\: I(x,x) < 0 \} \cup \{0\}.
\]
\end{defn}

\begin{thmA}[A global Morse index theorem of CMC form] \label{ThmA}
Given $M^{n}$ a CMC hypersurface immersed in $\mathbb{R}^{n+1}$ and a
$C^{0}$-monotone continuum $\mathcal{D} = \{ D(t) \:;\: t \in (0,b] \}$ of
Lipschitz domains in $M^{n}$. Denote by $H^{k}_{t} \equiv H^{k}(D(t))$ the
closure of $\mathcal{G}_{t}\equiv \mathcal{G}(D(t))$ in $W^{k,2}(D(t))$. When
$k = 1$, write $H^{1}_{t}$ by $H_{t}$. Consider the bilinear form
$\widetilde{I}$ = $\widetilde{I}_{t}$ of \eqref{e2.10} on $H_{t}$, and the
corresponding linear operator $\widetilde{L} \colon H^{2}_{t} \to H^{0}_{t}$,
given in \eqref{e2.9}. We have
\begin{equation} \label{e2.17}
  \widetilde{i}(r) = \sum_{0 \leq t < r} \widetilde{\nu}(t),
    \quad \forall\, r \in [0,b],
\end{equation}
where $\widetilde{\nu}(t)$ and $\widetilde{i}(t)$ are the corresponding nullity
and the index of $\widetilde{I}_{t}$ on $D(t)$, respectively.
\end{thmA}

\begin{rmk} \label{R2.12}
$u \in \operatorname{Ker} \widetilde{I}_{D(t)} \subset H_{t}$ means
$\widetilde{I}(u,\varphi) = 0$, $\forall\, \varphi \in H_{t}$, which says that
$u$ in $H_{t}$ is the weak solution of
\[ %\tag{2.25} \label{e2.25}
  \Delta_{M} \: u = -|B|^{2} u - \frac{1}{|D(t)|} \int_{D(t)} Lu.
\]
By the regularity argument as in Remark~\ref{R2.3}, $u$ is smooth, i.e., $u \in
G_{0}(D(t)) \subset C^{\infty}(D(t)) \cap C^{0}(\overline{D(t)})$ with
$\widetilde{L}u = 0$. Thus $u \in \operatorname{Ker} \widetilde{I}_{t}$ if and
only if $u$ is a Jacobi field on $D(t)$ in the classical sense. And
$\widetilde{\nu}(t)$ means the number of independent Jacobi fields on $D(t)$.
\end{rmk}

\begin{rmk} \label{R2.13}
We may regard in the Sobolev setting that every variation function $u \in
H_{t}$ on $D(t)$, as a variation function $\overline{u} \in H_{r}$ for $r > t$,
by extending $u$ to $D(r)$, letting $\overline{u} \equiv 0$ on $D(r) - D(t)$.
However, the extension $\overline{u}$ of a Jacobi field $u$ on $D(t)$ is not a
Jacobi field on $D(r)$ with $r > t$. In fact,
\[ %\tag{2.26} \label{e2.26}
  \widetilde{I}_{r}(\overline{u},\varphi)
  = \int_{D(t)} Du \cdot D\varphi - |B|^{2} u\varphi,
    \quad \forall\, \varphi \in H_{r},
\]
but the above integral is not necessarily zero, since $\varphi$ is not assumed
in $H_{t}$. Hence, $\overline{u} \notin \operatorname{Ker} I_{r}$ in general,
and therefore we can not replace the sum on the right-hand side of
\eqref{e2.17} by one single term $\widetilde{\nu}(r^{-})$ for $r^{-} < r$,
where $r^{-}$ is sufficiently near $r$.
\end{rmk}

For the proof of Theorem~A, we need the following strictness and continuity
theorem of the eigenvalues.

\begin{thmB} \label{ThmB}
Given $M^{n}$, $\mathcal{D}$ and $\widetilde{I}$ as in Theorem~\textup{A}, the
eigenvalues $\widetilde{\lambda}_{k} = \widetilde{\lambda}_{k}(t)$ of
$\widetilde{I}_{t}$ are continuous and ``strictly" decreasing in $t$, i.e.,
\[ %\tag{2.27} \label{e2.27}
  s < t \;\; \Rightarrow \;\;
  \widetilde{\lambda}_{k}(s) \gneqq \widetilde{\lambda}_{k}(t),
    \quad \forall\, s,t \in (0,b].
\]
\end{thmB}

The proof of Theorem~B is based on the Sobolev continuity (Theorem S in
\S\ref{S3.2}  ), which will be established later in \S\ref{S3}. Remark that for
the continuum of domains with Hausdorff continuity, the eigenvalues are not
necessarily continuous (see examples later in \S\ref{S3.1}, \S\ref{S3.2}).  \\

\begin{proof}[Proof of Theorem~\textup{A}]
Let Theorem~B be assumed. \emph{Step~$1$.} Given $\varphi \in H_{t}$, expand
\[ %\tag{2.28} \label{e2.28}
  \varphi = a_{1} \: \widetilde{u}_{1} + a_{2} \: \widetilde{u}_{2} + \cdots
\]
in $L^{2}(D(t))$ with an orthonormal eigenfunctions $\widetilde{u}_{k} \in
\mathcal {G}_{0}(D(t))$. Then for any $g \in H_{t}$,
\[ %\tag{2.29} \label{e2.29}
  \widetilde{I}_{t}(\varphi,g)
  = \langle \: a_{1} \: \widetilde{\lambda}_{1} \widetilde{u}_{1}
    + a_{2} \: \widetilde{\lambda}_{2} \widetilde{u}_{2} + \cdots, \; g
    \:\rangle.
\]
It is easy to see that
\begin{equation} \label{e2.18}
  \widetilde{I}_{t}(\varphi,g) = 0, \;\; \forall\, g \in H_{t},
  \quad \textrm{iff} \quad
  a_{i} = 0, \;\; \forall\, i \textrm{ with }
    \widetilde{\lambda}_{i}(t) \neq 0,
\end{equation}
by letting $g = \widetilde{u}_{i}$ for the only-if part. Or, equivalently,
$\varphi$ is a Jacobi field on $D(t)$ if and only if \eqref{e2.18} holds.
Hence, $\widetilde{\nu}(t)$ is the multiplicity of the zero eigenvalue of
$\widetilde{I}_{t}$.

\bigskip\noindent\emph{Step~$2$.} Let the eigenvalues $\widetilde{\lambda}_{j}$
on $D(t)$ be written in the form
\[ %\tag{2.31} \label{e2.31}
  \widetilde{\lambda}_{1}
  \leq \widetilde{\lambda}_{2}
  \leq \cdots
  \leq \widetilde{\lambda}_{k}
  < 0
  \leq \widetilde{\lambda}_{k+1}
  \leq \widetilde{\lambda}_{k+2}
  \leq \cdots
  \to \infty.
\]
When $D(s)$ enlarges from $0$ to $t$, $\widetilde{\lambda}_{1},
\widetilde{\lambda}_{2}, \ldots, \widetilde{\lambda}_{k}$ continuously passes
through zero, by Theorem~B, and becomes negative. We have $\sum_{0 \leq s < t}
\widetilde{\nu}(t) = k$, since as $s$ varies from $0$ to $t$, there occurs a
Jacobi field on $D(s)$, if and only if some $\widetilde{\lambda}_{j}(s) = 0$
for $j$ with $1 < j \leq k$, counting the multiplicity.

\bigskip\noindent\emph{Step~$3$.} On the other hand, given any nonzero $\varphi
= a_{1} \widetilde{u}_{1} + a_{2} \widetilde{u}_{2} + \cdots \in H_{t}$, we see
that $\varphi \in \widetilde{\Lambda}_{-}(t)$, if and only if
\[ %\tag{2.32} \label{e2.32}
  \widetilde{I}(\varphi,\varphi)
  = (\widetilde{\lambda}_{1} a^{2}_{1} + \cdots
    + \widetilde{\lambda}_{k} a^{2}_{k})
    + (\widetilde{\lambda}_{k+1} a^{2}_{k+1} + \cdots)
  < 0,
\]
where $\widetilde{\Lambda}_{-}(t)$ denotes the unstable cone in $H_{t}$, i.e.,
\[ %\tag{2.33} \label{e2.33}
  \widetilde{\Lambda}_{-}(t)
  := \{ f \in H_{t} \:;\: f = 0 \textrm{ or } \widetilde{I}(f,f) < 0 \}.
\]
Thus $\widetilde{W} :=$ the linear span $\langle \widetilde{u}_{1}, \ldots,
\widetilde{u}_{k} \rangle$ is a maximal linear subspace of
$\widetilde{\Lambda}_{-}(t)$ and
\[
  i(t) = \dim \widetilde{W} = k.
\]
This will complete the proof, if we have shown Theorem~B.
\end{proof}

\begin{rmk} \label{R2.14}
Theorems~A and B also hold for the Dirichlet case, where the stability operator
$L$ (see \eqref{e2.7} and \eqref{e2.8}) on $E_{t} \equiv E(D(t))$ is
considered. In fact, the proofs follow from the similar arguments of the more
complicated case with volume constraint, where $\widetilde{L}$ on $H_{t}$ is
considered.
\end{rmk}

%%========================================================= sec. 3 ==
\section{Sobolev continuity} \label{S3}
%%===================================================================
%%============================================ sec. 3.4 -> sec. 3.1==
\subsection{Examples} \label{S3.1}
%%===================================================================
\begin{ex}[Discontinuity of $\lambda_{k}(t)$] \label{E3.1}
Consider a monotone family $\mathcal{D}_{2} = \{ D(t) \subset \mathbb{R}^{1}
\:;\: t \in (0,1] \}$ defined by $D(t) = D_{-}(t) \cup D_{+}(t)$ for $0 < t <
1$, where
\[ %\tag{3.11} \label{e3.11}
  D_{-}(t) \equiv (-\pi,-\pi(1-t)), \quad
  D_{+}(t) \equiv (\pi(1-t),\pi),
\]
and $D(1) = (-\pi,\pi)$. We have $|B|^{2} = 0$ and $Lf = -\Delta_{M} f -
|B|^{2} f = -f''$. Given $t \in (0,1)$, the first eigenfunction $u_{1}(x)$ on
$D(t)$ is given by $\sin[\:(x - \pi(t-1))/t\:]$ on $D_{-}(t)$; and $\sin[\:(x +
\pi(t-1))/t\:]$ on $D_{+}(t)$. Clearly, $\lambda_{1}(t) = 1/t^{2}$ for $t \in
(0,1)$. However, at $t = 1$, $u_{1}(x) = \cos(x/2)$ with $\lambda_{1}(1) =
1/4$. Thus
\[ %\tag{3.12} \label{e3.12}
  \lim_{t \to 1^{-}} \lambda_{1}(t) = 1 \neq 1/4 = \lambda_{1}(1),
\]
and $\lambda_{1}(t)$ is not continuous at $t = 1$. Remark that the monotone
family $\mathcal{D}_{2}$ is continuous in the Hausdorff distance $d_{H}$ (see
\eqref{e2.2}), but not set-continuous (see \eqref{e2.4}). By
Definition~\ref{D2.7}, it is not a $C^{0}$-monotone continuum of domains in
$\mathbb{R}^{1}$.
\end{ex}

In general, a monotone family $\mathcal{D}$ satisfying the set-continuity is
continuous in the Hausdorff distance $d_{H}$. This will be seen in the proof of
Theorem~S. But the converse is not true (see Example~\ref{E3.3} later in
\S\ref{S3.2}).

\begin{ex} \label{E3.2}
Let $M^{2} = \{ (x,e^{iy}) \:;\: x,y \in \mathbb{R}^{1} \} \subset \mathbb{R}
\times S^{1}$ be a cylinder in $\mathbb{R}^{3}$. Let $\mathcal{D}_{3} = \{ D(t)
\:;\: t \in [a,\infty) \}$, where $a$ is a small number, $p_{0} = (0,1) \in
M^{2}$, $D(a) =$ a small neighborhood of $p_{0}$, and $D(t) =$ a geodesic disk
centered at $p_{0}$ with radius $t$ for $t \in (a,\pi)$. When $t = \pi$,
$D(\pi)$ is the geodesic disk with its boundary $\partial D(\pi)$ gluing
together the two antipodal points at $-p_{0}$. After $t > \pi$, $D(t)$ is the
geodesic disk, which overlaps itself  in a larger area (see Figure~\ref{F6}).
\begin{figure}[H]
\begin{minipage}[t]{0.45\textwidth}
\centering
\includegraphics{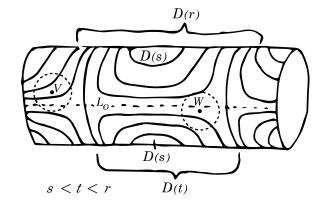}
\caption{(also related to Theorem~S and Remark~\ref{R4.1})} \label{F5}
\end{minipage} \quad
\begin{minipage}[t]{0.5\textwidth}
\centering
\includegraphics{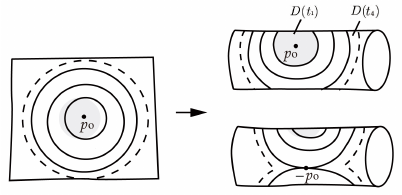}
\caption{$D(t_{1}) \subset \cdots \subset D(t_{4})$, $D(t_{1}) =$ disk, $i =
1,2,3$; $D(t_{4}) =$ ring} \label{F6}
\end{minipage}
\end{figure}
\end{ex}

The continuum $\mathcal{D}_{3}$ is not only Hausdorff continuous in $d_{H}$,
but also satisfies the set-continuity. According to Definition~\ref{D2.7},
$\mathcal{D}_{3}$ is a $C^0$-continuum of Lipschitz domains in $M^{n}$. One of
the main results, i.e., Theorem~J in \S\ref{S5.2}, is  valid for
$\mathcal{D}_{3}$. Note that the topological type of $D(t)$ \emph{changes} from
a disk to a ring domain.

%%=========================================== sec. 5.1 -> sec. 3.2 ==
\subsection{Sobolev continuity} \label{S3.2}
%%===================================================================
%\subsection{Examples of continuums} \label{S5.1}
%%===================================================================
Following the formulation of Smale \cite{S65}, Frid--Thayer \cite{FT90} had
proved a Morse index theorem in an abstract form for a given smooth deformation
of nested domains $D(t)$ in $M^{n}$. The restrictions of their formulation are
two folded: (i)~$\partial D(t)$ must be smooth, thus many interesting examples
are not included; (ii)~the topological type of $D(t)$ can not be changed as $t$
varies.\\

We prove the Morse index theorem (see Theorem~A) and its application to Jacobi
fields (see Theorem~J in \S\ref{S5.2}) for the wider class of continuums, i.e.,
$C^{0}$-monotone continuums (see Definition~\ref{D2.7}), in which the two
restrictions will be waived. More precisely, we only require each $\partial
D(t)$ as a boundary of a generalized Lipschitz domain. And we allow the
topological type of $D(t)$ to change as $t$ varies.\\

\begin{wrapfigure}{r}{0.3\textwidth}
\centering
\includegraphics[width=0.25\textwidth]{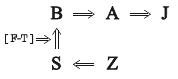}
\caption{logical chart of the main theorems} \label{F7}
\end{wrapfigure}
In order to show Theorem~B for $C^{0}$-monotone continuums. We first prove the
Sobolev continuity \eqref{e3.1}, stated in Theorem~S (see below). Then use the
Sobolev continuity to show the continuity of eigenvalues and the strictness,
i.e., Theorem~B.  These will be presented in \S\ref{S3} and \S\ref{S4}. In
\S\ref{S5} we follow the chart in Figure~\ref{F7} to establish a distribution
theorem of Jacobi fields, i.e., Theorem~J. We remind that Theorem~S is the core
and difficult result of the paper. However, we need to establish the three
closure theorem (see Theorem~Z) to prove Theorem~S. By the logical chart in
Figure~\ref{F7}, we mean for example Theorem~S is based on Theorem~Z, Theorem~J
based on Theorem~A, and so on.

\begin{thmS}[Sobolev continuity] \label{ThmS}
For a $C^{0}$-monotone continuum $\mathcal{D} = \{ D(t); t \in [0,b] \}$ of
generalized Lipschitz domains $D(t)$ in $M^{n}$, we have
\begin{equation} \label{e3.1}
  \overline{\bigcup_{s < t} H_{s}} = H_{t} = \bigcap_{r > t} H_{r},
\end{equation}
where $H_{t} \equiv H(D(t))$, $\forall\, t \in [0,b]$. The formula~\eqref{e3.1}
is called Sobolev continuity.
\end{thmS}

\begin{defn} \label{D3.1}
A monotone family $\mathcal{D}$ of domains $D(t)$ in $M^{n}$ is called a
\emph{smooth monotone continuum}, if there is a $C^{\infty}$-map $\alpha \colon
\overline{D} \times [0,b] \to M^{n}$, such that for any given $t \in [0,b]$,
$D(t) = \alpha(D \times \{t\})$, $\alpha \big|_{D \times \{t\}}$ is a
diffeomorphism with $\partial D(t)$ a $C^{\infty}$-submanifold of $M^{n}$.
\end{defn}

\begin{rmk} \label{R3.2}
Smooth monotone continuums are the subject that Frid--Thayer treated. For that
case, the claim~\eqref{e3.1} can be shown directly by using the
$C^{\infty}$-map $\alpha$ to shrink $D(r)$. For example, the second equality of
\eqref{e3.1} is proved like this: Given $g \in \bigcap_{r > t} H_{r}$, for each
$r > t$,  $\exists\ f \in \mathcal{G}_{0}(D(r))$ very close to $g$ in $H_{r}
\subset H_{b}$. Shrinking $D(r)$ through $\alpha$ to $D(t)$, we modify $f \in
\mathcal{G}_{0}(D(r))$ to obtain $h \in \mathcal{G}_{0}(D(t))$. Namely, define
$\beta \colon D(r) \xrightarrow{\approx} D(t)$ by $\beta(y) = \alpha_{t}(\; p
\circ \alpha^{-1}_{r}(y),t) \in D(t)$ for each $y \in D(r)$, where $\alpha_{r}
\equiv \alpha \big|_{D \times r}$ and $p$ is the projection of $\overline{D}
\times [0,b]$ onto $\overline{D}$. Define $h(x)=f\circ \beta^{-1}(x)$ for each
$x \in D(t)$, and $h(x) \equiv 0$ on $D(r)-D(t)$.  By choosing $r$ very close
to $t$, $h$ is very close to $f$ in $H_{r} \subset H_{b}$. Thus $g$ is in the
closure of $\mathcal{G}_{0}(D(t))$ in $H_{b}$, and therefore $g \in H_{t}$.
Applying this kind of shrinking arguments, the rest of the proof is easily
obtained.\\

However, for $C^{0}$-monotone continuums, the story is entirely different,
since $\alpha$ is not available. The notion of set-continuity that we
introduced is now significant to establish the Sobolev continuity.
\end{rmk}

\begin{ex} \label{E3.3}
Consider $M^{2} \equiv \mathbb{R} \times (-\pi,\pi)$ and a monotone family
$\mathcal{D}_{4} \equiv \{ D(t) \subset M^{2} \:;\: 0 < t \leq b \}$ (see
Figure~\ref{F8}), given by: (i)~$D(1) = (-1,1) \times (-\pi,\pi)$; (ii)~For $s
< 1$, $D(s)$ is a union of two components of half-disk shape, so that each
approaches to rectangles $(-1,1) \times (-\pi,0)$ and $(-1,1) \times (0,\pi)$
as $s \to 1$; (iii)~For $r > 1$, $D(r)$ is the rectangle $(-r,r) \times
(-\pi,\pi)$. Clearly, $\mathcal{D}_{4}$ is continuous in Hausdorff distance
$d_{H}$, while not satisfying set-continuity at $t = 1$. In fact, $\bigcup_{s <
1} D(s) \subsetneqq D(1)$, the condition~\eqref{e2.4} is not satisfied,
although \eqref{e2.5} is valid. If we define $\mathcal{D}'_4$ the same as
$\mathcal{D}_4$, except letting $D(1) = ((-1,1) \times (-\pi,\pi)) - L_{0} =
\bigcup_{s < 1} D(s)$ instead, where $L_{0} = (-1,1) \times \{0\}$, then
\eqref{e2.4} is satisfied for $\mathcal{D}'_4$. However, $\partial(D(1))
\subsetneqq \partial(\overline{D(1)})$, \eqref{e2.5} is not true. Hence, both
$\mathcal{D}_4$ and $\mathcal{D}'_4$ are not $C^{0}$-monotone continuums.
\end{ex}

\begin{wrapfigure}{r}{0.5\textwidth}
\centering
\includegraphics[width=0.4\textwidth]{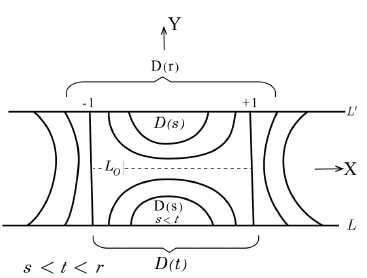}
\caption{$M^{2} = \mathbb{R} \times (-\pi,\pi)$} \label{F8}
\end{wrapfigure}
Intuitively, $\overline{\bigcup_{s < 1} H_{s}} \subsetneqq H_{1}$ for
$\mathcal{D}_{4}$ in Example~\ref{E3.3}, since a Sobolev function $g$ in
$H_{1}$ may not vanish around the middle line $L_{0} \equiv (-1,1) \times 0$,
while the trace of each $f$ in $H_{s}$ on $\partial D(s)$ is zero, $\forall\, s
< 1$. If we project $\mathcal{D}_{4}$ along $X$-axis to the interval
$(-\pi,\pi)$, $\mathcal{D}_{4}$ becomes $\mathcal{D}_{2}$ of
Example~\ref{E3.1}, of which the eigenvalues of $L$ is discontinuous at $t =
1$. Thus, Theorem~B is not expected to be valid for $\mathcal{D}_{4}$ as well.

\begin{ex} \label{E3.4}
One may reconstruct $\mathcal{D}_{4}$ and $\mathcal{D}'_4$ by identifying the
two lines $L := \mathbb{R} \times (-\pi)$ and $L' := \mathbb{R} \times \pi$, to
obtain monotone families $\mathcal{D}_{5}$ and $\mathcal{D}'_5$ respectively,
on the cylinder $M^{2} \equiv \mathbb{R} \times S^{1} = \{ (x,,e^{iy}) \:;\:
x,y \in \mathbb{R}^{1} \}$. Both $\mathcal{D}_{3}$ and $\mathcal{D}_{5}$ are
two “similar” monotone families on the cylinder $\mathbb{R} \times S^{1}$
which change topological type of $D(t)$ from disk to cylinder. However,
$\mathcal{D}_{3}$ is a $C^{0}$-monotone continuum, while $\mathcal{D}_{5}$ is
not. Notice that Theorem~B applies to $\mathcal{D}_{3}$, but not to
$\mathcal{D}_{5}$, nor to $\mathcal{D}'_5$, $\mathcal{D}_4$ and
$\mathcal{D}'_4$.
\end{ex}

%%=========================================== sec. 5.3 -> sec. 3.3 ==
\subsection{Three closure theorem} \label{S3.3}
%%===================================================================
We need the following three closure theorem to prove the Sobolev continuity
later in \S\ref{S4}.  Given $D$ a generalized Lipschitz domain in $M^{n}$ (see
Definition~\ref{D2.2}), satisfying the set-continuity (see Definition~\ref{D2.7}), i.e., $\partial\overline{D}=\partial{D}$. Clearly, the joint set of $D$ (if
any) contains no open set of some $\Gamma_{i}$ in a triple $(U_{i},\Gamma_{i},V_{i})$ of any $\Tilde{U_p}$ (see Definition~\ref{D2.2}). Let
\begin{align*}
  \mathcal{F}_{m}(D)
  &:= \{ f \in C^{\infty}(D) \cap C^{m}(\overline{D}) \:;\: f|_{\partial D} = 0
    \}, \\
%\tag{5.3} \label{e5.3}
  \mathcal{F}_{c}(D)
  &:= \{ f \in C^{\infty}(D) \:;\: \overline{\operatorname{supp} f} \subset D
    \},
\end{align*}
        where $m$= 0 or 1. Then $\mathcal{F}_{0}(D) \supset \mathcal{F}_{1}(D)
        \supset \mathcal{F}_{c}(D)$. Although $\mathcal{F}_{0}(D)\subsetneqq
        W^{1,2}(D)$ (for example letting $f\in \mathcal{F}_{0}(0,1)$ with
        $f(x)=\sqrt{x}$ on $[0,\delta),\; 0<\delta< 1$), it is clear that
        $\mathcal{F}_{1}(D)$ and $\mathcal{F}_{c}(D)$ are contained in
        $W^{1,2}(D)$. If we consider the three subspaces
        $\mathcal{F}_{0}(D)\cap W^{1,2}(D), \mathcal{F}_{1}(D)$ and
        $\mathcal{F}_{c}(D)$ of $W^{1,2}(D)$,
it is interesting to observe that their closures in $W^{1,2}(D)$ are the
same.\\

For the case of volume constraint, the parallel result is also true. It seems
that the three closure theorem must be known in literature, but we find no
handy reference. However, we will give here an intuitive direct proof, in which
some interesting ideas are presented. Not only the result of the three closure
theorem is repeatedly used in the sequel, some crucial estimates of its proof
will also be modified to show the Sobolev continuity.
\\

\begin{thmZ}[Three closure theorem] \label{ThmZ}
    Given $D$ a generalized Lipschitz domain in a Riemannian manifold $M^{n}$,
    let $D$ satisfies the set-continuity, i.e., $\partial \overline{D}
    =\partial D$. Then the closures of $\mathcal{F}_{0}(D)\cap W^{1,2}(D),
    \mathcal{F}_{1}(D)$ and $\mathcal{F}_{c}(D)$ in $W^{1,2}(D)$ are the same
    space $E(D)$. Namely,
\begin{align}
%\tag{5.4} \label{e5.4}
  E(D) &= W_{0}^{1,2}(D), \notag \\
\label{e3.2}
  H(D) &= \textrm{the closure of $\mathcal{G}_{c}(D)$ in $W^{1,2}(D)$},
\end{align}
where $E(D)$ and $H(D)$ are given in \eqref{e2.1}, $W_{0}^{1,2}(D)$ denotes the
closure of $\mathcal{F}_{c}(D)$ in $W^{1,2}(D)$ and
\[
  \mathcal{G}_{c}(D)
  := \big\{ g \in C^{\infty}(D) \:;\: \overline{\operatorname{supp} g}
    \subset D, \; \int_{D}g = 0, \big\}.
\]
\end{thmZ}

Remark that $W_{0}^{1,2}(D)$ is also identical with the space of all the
Sobolev functions in $W_{0}^{1,2}(D)$, having zero trace on $\partial D$. Also
note that in the theorem, $D$ may have joint points, where some of its boundary
points are gluing together. Yet the joint set is not open in any $\Gamma_{i}$.
(see Definition 2.2). \\

We shall prove Theorem~Z first locally on a standard Lipschitz simple domain
$(U,\Gamma,V)$ and then extend it to the global version  by a partition of
unity on the given  generalized Lipschitz domains. We have to  treat carefully
the behaviors of variation functions around the possible joint points. In the
sequel, a Lipschitz simple domain $(U,\Gamma,V)$ will be regarded as in
$\mathbb{R}^{n}$, as well as in $M^{n}$, by bridging them freely with a local
coordinate map.

Given $(U,\Gamma,V)$, we may shrink $U$ and $V$ such that
\begin{equation} \label{e3.3}
  U = \{ (x,r) \in \mathbb{R}^{n} \:;\: x \in V, u(x) < r < v(x) \},
\end{equation}
$\Gamma = \{ (x,u(x)) \in \mathbb{R}^{n} \:;\: x \in V \}$, and $V = \{ (x,0)
\in \mathbb{R}^{n} \:;\: x \in V \subset H^{n-1} \}$, where $(x,r)$ is the
chosen coordinate of $U$ and $H^{n-1}$ is a $(n-1)$-hyperplane of
$\mathbb{R}^{n}$. Let $L > 0$ be the Lipschitz constant of $u(x)$, $\forall\, x
\in V$. If we pull $(U,\Gamma,V)$ back to $M^{n}$ with a coordinate map of
$M^{n}$, we may assume that $r$-coordinate curves are geodesics in $M^{n}$
perpendicular to $V$. We call such a $(U,\Gamma,V)$ \emph{standard}, and assume
hereafter that any Lipschitz simple domain under consideration is standard,
unless otherwise specified.

\begin{lemA}[Local version of Theorem~Z] \label{LemA}
Given a (standard) Lipschitz simple domain $(U,\Gamma,V)$, and $f \in
C^{\infty}(U) \cap C^{0}(\overline{U} \cap \Gamma)$ such that $f|_{\Gamma} = 0$
and $f \in W^{1,2}(U)$. Let
\[ %\tag{5.7} \label{e5.7}
  C_{\Gamma}^{\infty}(U)
  := \{ h \in C^{\infty}(U) \:;\: \overline{\operatorname{supp} h} \cap \Gamma
    = \phi \big\}.
\]
Then $\forall\, \varepsilon > 0$, $\exists\, h \in C_{\Gamma}^{\infty}(U)$ such
that
\begin{equation} \label{e3.4}
  \|f-h\|_{W^{1,2}(U)}^{2} < \varepsilon.
\end{equation}
\end{lemA}
\begin{proof}
\emph{Step~$1$}. Let $\overline{\delta} > 0$, define
\begin{equation} \label{e3.5}
  N_{\overline{\delta}}
  := \{ (x,r) \in U \:;\: x \in V, u(x) < r < u(x)+\overline{\delta} \}.
\end{equation}

Select $\overline{\delta} = 4\delta$. Consider $\overline{u} = u+2\delta$ and
the smooth approximation function $w(x)$ (see Figure~\ref{F9}) given by
\begin{equation} \label{e3.6}
  w(x)
  := \int_{V} \overline{u}(y) \varphi_{\alpha}(y-x) \, dy, \quad x \in V,
\end{equation}
\begin{figure}[H]
\centering
\includegraphics[width=0.45\textwidth]{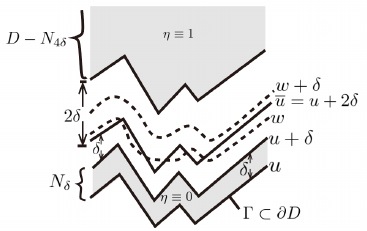}
\caption{} \label{F9}
\end{figure}
where $\alpha$ is a small number,
\[ %\tag{5.11} \label{e5.11}
  \varphi_{\alpha}(x)
  := \frac{1}{\alpha^{n-1}} \varphi \big( \frac{x}{\alpha} \big),
\]
and $\varphi \in C^{\infty}(H^{n-1})$ is the standard mollifier. We may freely
expand $U$, $V$ and $u(x)$ a little outward to make $w(x)$ well-defined around
$\partial V$ in $H^{n-1}$. For $x \in V$, we see that
\begin{equation} \label{e3.7}
\begin{split}
  |w(x)-\overline{u}(x)|
  &= \big| \int_{V} (\overline{u}(y)-\overline{u}(x))\; \varphi_{\alpha}(y-x)
  \, dy \big| \\
  &\leq \int_{B_{\alpha}(x)} L \cdot |y-x| \;\varphi_{\alpha}(y-x) \, dy
  \leq L \cdot \alpha
  = \delta,
\end{split}
\end{equation}
by choosing $\alpha = \delta/L$.

\bigskip\noindent\emph{Step~$2$}. Define $\eta \in C^{\infty}(U)$ by $\eta(x,r)
:= \eta_{0} \big( \frac{r-w(x)}{\delta} \big)$, where $\eta_{0} \in
C^{\infty}(\mathbb{R})$ is a cut-off function with $\eta_{0}(x) \equiv 1$ for
$x > 1$; $0 \leq \eta_{0}(x) \leq 1$ for $0 \leq x \leq 1$; and $\eta_{0}(x)
\equiv 0$ for $x \leq 0$. Clearly, $\eta \equiv 1$ on $U -
N_{\overline{\delta}}$, and $\eta \equiv 0$ on $N_{\delta}$, recalling
$\overline{\delta} = 4\delta$. Define
\[ %\tag{5.13} \label{e5.13}
  h := f \cdot \eta \in C_{\Gamma}^{\infty}(U),
\]
we estimate the Sobolev norm of $f-h$ as follows:
\begin{equation} \label{e3.8}
  \int_{U} |f-h|^{2}
  = \int_{N_{\overline{\delta}}} |f(1-\eta)|^{2}
  \leq \int_{N_{\overline{\delta}}} f^{2}
  \leq C \big( \int_{N_{\overline{\delta}}} |Df|^{2} \big) \cdot \delta^{2},
\end{equation}
where the last inequality will be shown by Lemma~B right after this proof of
Lemma~A, and
\begin{equation} \label{e3.9}
\begin{split}
  \int_{U} |Df-Dh|^{2}
  &= \int_{N_{\overline{\delta}}} |Df(1-\eta) - f(D\eta)|^{2} \\
  &\leq C \big( \int_{N_{\overline{\delta}}} |Df|^{2} + f^{2} |D\eta|^{2}
  \big).
\end{split}
\end{equation}
Hereafter $C$ in formulas denotes positive constants of any form, independent
of $\delta$, $(x,r)$, and $f$, $h$, $\ldots$, such as $C = 2C = 2 = L =
1+L^{2}$, etc., just for convenience. Note that $\int_{U} |Df|^{2} < \infty$,
since $f \in W^{1,2}(U)$. Write $P(\overline{\delta}) \equiv
\int_{N_{\overline{\delta}}} |Df|^{2}$, then $P(\overline{\delta}) \to 0$ as
$\overline{\delta} \to 0$. Let the last term of \eqref{e3.9} be denoted by
\begin{equation} \label{e3.10}
  W := \int_{N_{\overline{\delta}}} f^{2} \cdot |D\eta|^{2},
\end{equation}
which is the difficult part to estimate. In the following Step~3, we will
establish $|D\eta|^{2} \leq C/\delta^{2}$. Then by Lemma~B, we have
\begin{equation} \label{e3.11}
  W
  \leq \frac{C}{\delta^{2}} \big( \int_{N_{\overline{\delta}}} f^{2} \big)
  \leq \frac{C}{\delta^{2}} (P(\overline{\delta}) \cdot \overline{\delta}^{2})
  = C \cdot P(\overline{\delta})
  \to 0,
\end{equation}
as $\delta \to 0$. Given $\varepsilon > 0$, we may then choose $\delta > 0$
small such that
\[ %\tag{5.18} \label{e5.18}
  \|f-h\|_{W^{1,2}(U)}^{2}
  = \int_{U} |f-h|^{2} + \int_{U} |Df-Dh|^{2}
  < \varepsilon,
\]
as required.

\bigskip\noindent\emph{Step~$3$}. We have
\begin{equation} \label{e3.12}
\begin{split}
  |D\eta|^{2}
  & \leq 2(|D_{x} \eta|^{2} + ( \frac{\partial \eta}{\partial r} )^{2}) \\
  &= 2 \big[ \eta'_{0} \big( \frac{r-w(x)}{\delta} \big) \big]^{2}
    \big[ \big( \frac{-1}{\delta} |D_{x}w| \big)^{2}
    + \big( \frac{1}{\delta} \big)^{2} \big] \\
  &\leq 2 \frac{C_{1}}{\delta^{2}} (1+|D_{x}w|^{2}),
\end{split}
\end{equation}
where $C_{1}$ is the bound of the term $|\eta'_{0}|^{2}$, a constant
independent of $\delta$. In the first inequality of \eqref{e3.12}, the number 2
is attributed to the deviation of the metric $g_{ij}$ of $M^{n}$ from the flat
metric on $(x,r)$ of $(U,\Gamma,V)$. Claim that
\begin{equation} \label{e3.13}
  |D_{x}w| \leq L.
\end{equation}
Given $x \in V$, for $z$ very close to $x$ in $V$, we have
\begin{align*} %\tag{5.21} \label{e5.21}
  &\quad |w(z)-w(x)| \\
  &= \big| \int_{V} \overline{u}(y) \varphi_{\alpha}(y-z) \, dy
    - \int_{V} \overline{u}(y) \varphi_{\alpha}(y-x) \, dy \big| \\
  & = \big| \int_{B_{\alpha}(0)} \overline{u}(z+\zeta')
  \varphi_{\alpha}(\zeta')
    \, d\zeta'
    - \int_{B_{\alpha}(0)} \overline{u}(x+\zeta) \varphi_{\alpha}(\zeta) \,
    d\zeta \big| \\
  & \leq \big| \int_{B_{\alpha}(0)}
  |\overline{u}(z+\zeta)-\overline{u}(x+\zeta)|\;
    \varphi_{\alpha}(\zeta) \, d\zeta \\
  &\leq L \cdot |z-x|.
\end{align*}
Letting $z \to x$, we obtain \eqref{e3.13} and hence
\begin{equation} \label{e3.14}
  |D\eta|^{2} \leq 2C_{1} (1+L^{2})/\delta^{2} = C/\delta^{2}.
\end{equation}
\end{proof}

\begin{lemB} \label{LemB}
Let $(U,\Gamma,V)$ be a Lipschitz simple domain given by \eqref{e3.3}. For any
$\delta > 0$, let $N_{\delta} \subset U$ be the $\delta$-neighborhood of
$\Gamma$ in $U$ defined as in \eqref{e3.5}. Given $f$ as in Lemma A, we have
\begin{equation} \label{e3.15}
  \int_{N_{\delta}} f^{2}
  \leq \big( \int_{N_{\delta}} |Df|^{2} \big) \cdot \frac{\delta^{2}}{2}
\end{equation}
with $P(\delta) := \int_{N_{\delta}} |Df|^{2} \to 0$ as $\delta \to 0$.
\end{lemB}
\begin{proof}
Write $f = f(x,r)$, $x \in V$ and $u(x) < r < u(x)+\delta$. Let $\overline{r} =
r-u(x)$, then
\begin{equation} \label{e3.16}
\begin{split}
  \int_{N_{\delta}} f^{2} \, dr dx
  &= \int_{V} \big( \int_{0}^{\delta} f(x,\overline{r})^{2} \, d\overline{r}
  \big)
    \, dx \\
  &= \int_{V} \int_{0}^{\delta} \big( \int_{0}^{\overline{r}} f_{r}(x,s)
    \, ds \big)^{2} \, d\overline{r}\; dx
  \equiv I,
\end{split}
\end{equation}
where $f_{r} = \frac{\partial f}{\partial r}$. But
\begin{equation} \label{e3.17}
  \big( \int_{0}^{\overline{r}} f_{r}(x,s) \, ds \big)^{2}
  \leq \big( \int_{0}^{\overline{r}} f_{r}^{2} \, ds \big) \cdot \overline{r}
  \leq \big( \int_{0}^{\overline{r}} |Df|^{2} \, ds \big) \;\overline{r}.
\end{equation}
Thus
\begin{equation} \label{e3.18}
\begin{split}
  I
  &\leq \int_{V} \int_{0}^{\delta} \big[ \big( \int_{0}^{\delta} |Df|^{2}
    \, ds \big) \cdot \overline{r} \big] \, d\overline{r}\; dx \\
  &= \int_{V} \big( \int_{0}^{\delta} |Df|^{2} \, ds \big)
    \big( \int_{0}^{\delta} \overline{r} \, d\overline{r} \big) \, dx \\
  &= \int_{V} \big( \int_{0}^{\delta} |Df|^{2} \, ds \big) \, dx
    \cdot \frac{\delta^{2}}{2}
  = \big( \int_{N_{\delta}} |Df|^{2} \big) \cdot \frac{\delta^{2}}{2}.
\end{split}
\end{equation}
Remark that the coordinate transformation $(x,r) \mapsto (x,\overline{r})$ is
not a diffeomorphism. However, as only Fubini's theorem is involved in the
integration, the diffeomorphism is not required in the above estimation.
\end{proof}

The estimation of $W = \int_{N_{\delta}} f^{2} \cdot |D\eta|^{2}$ (see
\eqref{e3.10}) is crucial to the proof of Theorem~Z. For the local version, it
is interesting to preview its estimation on $1$-dimensional domains, i.e., $n =
1$. For example, consider $U$ an interval $(0,a) \subset \mathbb{R}^{1}$, and
$f(r) = r^{\alpha}$, $r \in (0,a)$ and $1/2 \leq \alpha < 1$. Then $\lim_{r \to
0} \frac{\partial f}{\partial r} = \infty$. Clearly, $f \in W^{1,2}(U)$ if and
only if $1/2 < \alpha$. Through straightforward computations, one sees that $W
= C\delta^{1/3}$ for $f(r) = r^{2/3}$, and $W = C\delta^{1/5}$ for $f(r) =
r^{3/5}$. As for $\alpha = 1/2$, i.e., $f(r) = r^{1/2}\notin W^{1,2}(U)$, and
$W$ is not controllable, since
\[
  \int_{0}^{\delta} f_{r}^{2} \, dr
  = \frac{1}{4} \int_{0}^{\delta} \frac{1}{r} \, dr
  = \infty,
\]
and
\[ %\tag{5.27} \label{e5.27}
  W
  = \int_{0}^{\delta} f^{2} \cdot |D\eta|^{2} \, dr
  = \big( \int_{0}^{\delta} r \, dr \big) \cdot \frac{C}{\delta^{2}}
  = \frac{C}{2},
\]
which is a nonzero constant, not tending to zero as $\delta \to 0$.\\

\begin{wrapfigure}{r}{0.3\textwidth}
\centering
\includegraphics[width=0.2\textwidth]{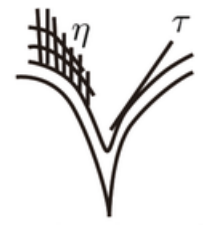}
\caption{$\tau$ tends to be vertical.} \label{F10}
\end{wrapfigure}
On the other hand, when $n \geq 1$ the shape of $\Gamma$ which is given on the
boundary of the Lipschitz domain plays the role for estimating $W$. Consider $n
= 2$, $V = (-1,1)$. Suppose $\Gamma$ is not Lipschitz, for example, let $u(x) =
|x|^{\beta}$ with $0 < \beta < 1$ so that $\Gamma$ has a cusp at $x = 0$. Then
$|D_{x}w|$ is not bounded by a constant $L$ and may tend to $\infty$ as $\delta
\to 0$, in which case, the estimation $|D\eta|^{2} \leq C/\delta^{2}$ is not
valid. The geometry (see Figure~\ref{F10}) is that the tangent line $\tau$ of
$r = w(x)$ becomes almost parallel to the $r$-curves as $\delta \to 0$, which
makes $|D_{x}w|$ blowing up and $|D_{x} \eta|^{2}$ bigger than $C/\delta^{2}$
in the sense of order. However, if $\Gamma$ is Lipschitz, the Lipschitz
conditions of $u(x)$ forces the tangent line $\tau$ and $r$-curves to keep at
least a constant angle away from zero. A close examination on \eqref{e3.12}
explains the last comment.

Thus, the assumption that the deforming domains are at least Lipschitz is
essential to establish the three closure theorem.

%%=========================================== sec. 5.4 -> sec. 3.4 ==
\subsection{Global version} \label{S3.4}
%%===================================================================
Now we use a partition of unity to patch together the results of Lemma~A on
local $(U_{i},\Gamma_{i},V_{i})$. \\

\begin{proof}[Proof of Theorem~\textup{Z}]
First consider the case of no volume constraint: Given $g$ in the closure of
$\mathcal{F}_{0}(D)\cap W^{1,2}(D)$ in $W^{1,2}(D)$, i.e., $\forall\epsilon >0,  \exists
f \in \mathcal{F}_{0}(D)\cap W^{1,2}(D)$ such that $\|g-f\|_{W^{1,2}(D)}^{2} <
\varepsilon$, we claim that there exists $h \in \mathcal{F}_
{c}(D)$ with $\|g-h\|_{W^{1,2}(D)}^{2} < \varepsilon$, by making $h$ very close
to $f$ in $W^{1,2}(D)$.

\bigskip\noindent\emph{Step~$1$}.
 For any $p \in \partial D$, choose an open set $\widetilde{U}_{p}$ of $M^{n}$
 with $p \in \widetilde{U}_{p}$.
The compact boundary $\partial D$ has a finite subcover $\{
\widetilde{U}_{p_{1}}, \ldots, \widetilde{U}_{p_m} \}$ of $\{ \widetilde{U}_{p}
\:;\: p \in \partial D \}$ with each $p_{k} \in \partial D$. As $D$ is a
“generalized” Lipschitz domain, it may happen that some particular
$\widetilde{U}_{p_{k}} \cap D$ is a disjoint union of more than one components,
each being a Lipschitz simple domain (see Figure~\ref{F4}). Consider all the
components of $\widetilde{U}_{p_{k}} \cap D$ with $k$ ranging over all $\{
1,2,\ldots,m \}$, and rewrite them as $\{ U_{1},\ldots,U_{\nu} \}$. By
shrinking $\widetilde{U}_{p}$ in the beginning, we may assume each
$(U_{i},\Gamma_{i},V_{i})$ a standard Lipschitz simple domain, where
$\Gamma_{i}$ is attached to $U_i$. Remark that if the joint set has a part $J$
which is open in some $\Gamma_{i}$, the given $f$ is not necessarily zero on
the whole $\Gamma_i$, since $J$ is not included in $\partial D$. However, we have
assumed that the generalized Lipschitz domain $D$ satisfies the set-continuity,
the joint set can not be open in $\Gamma_{i}$. Hence, $f=0$ on each
$\Gamma_{i}$, and the arguments of Lemma A are applicable here on each
$(U_{i},\Gamma_{i},V_{i})$.

\bigskip%
%\begin{wrapfigure}{r}{0.4\textwidth}
\begin{figure}[h]
\centering
\includegraphics[width=0.35\textwidth]{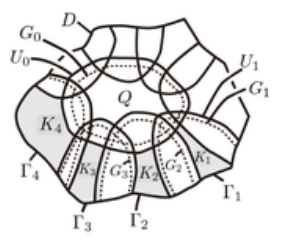}
\caption{constructing partition of unity} \label{F11}
\end{figure}
%\end{wrapfigure}
\bigskip\noindent\emph{Step~$2$.} Let $Q := D - (U_{1} \cup \cdots \cup
U_{\nu})$. Choose $U_{0}$ open in $D$ with $Q \subset U_{0} \subset\subset D$.
Write $U'_{i} := U_{i} \cup \Gamma_{i}$ for $i = 1,\ldots,\nu$. Then $\{
U_{0},U_{1},\ldots,U_{\nu} \}$ is an open cover of $D$, and $\{
U_{0},U'_{1},\ldots,U'_{\nu} \}$ is a cover of $\overline{D}$. Construct a
partition of unity by shrinking each $U_{i}$ once again. First we define an
extended notion of compactly inclusion up to the boundary, $``\subset\subset'
"$, as follows. Given $A$, $B$ open in $D$. If there are extensions
$\widetilde{A}$ and $\widetilde{B}$, both open in $M^{n}$ with $\widetilde{A}
\cap D = A$, $\widetilde{B} \cap D = B$ and $\widetilde{A} \subset\subset
\widetilde{B}$, we say that $``A \subset\subset' B"$, which means ``$A
\subset\subset B$ except on $\partial D$". Let $K_{1} := D - (U_{2} \cup U_{3}
\cup \cdots \cup U_{\nu} \cup U_{0}) \subset U_{1}$. Consider $G_{1}$ open in
$U_{1}$ such that $K_{1} \subset G_{1} \subset\subset' U_{1}$. Similarly, let
$K_{2} := D - (G_{1} \cup U_{3} \cup U_{4} \cup \cdots \cup U_{\nu} \cup U_{0})
\subset U_{2}$ and consider $G_{2}$ open in $U_{2}$ such that $K_{2} \subset
G_{2} \subset\subset' U_{2}$. Inductively, find $G_{i} \subset\subset' U_{i}$
for any $i = 1,2,3,\ldots,\nu$. Finally, choose an open set $G_{0}
\subset\subset D$ and $G_{0} \supset D - (G_{1} \cup \cdots \cup  G_{\nu})
\supset U_{0}$. We obtain another open cover $\{ G_{0},G_{1},\ldots,G_{\nu} \}$
of $D$ (see Figure~\ref{F11}). Define $\psi_{i} \in C^{\infty}(D)$ for each $i
= 1,2,\ldots,\nu$, such that $\psi_{i} \equiv 1$ on $G_{i}$, $\psi_{i} \equiv
0$ on $D-U_{i}$ and $0 \leq \psi_{i} \leq 1$ on $D$. And let $\psi_{0}\in
C^{\infty}_{c}(D)$, such that $\psi_{0}\equiv$ 1 on $G_{0}$; $0\leq
\psi_{0}\leq 1$ on $D$.    We see that $\forall\, p \in D$, $\exists\,
\psi_{i}$ with $\psi_{i}(p) = 1$, since $\{ G_{0},G_{1},\ldots,G_{\nu} \}$
covers $D$. Hence $\psi_{0} + \psi_{1} + \cdots + \psi_{\nu} \gneqq 0$ on $D$.
Consider $\varphi_{i} := \psi_{i}/(\psi_{0} + \psi_{1} + \cdots + \psi_{\nu})
\in C^{\infty}(D)$. We obtain that
\[ %\tag{5.28} \label{e5.28}
  1_{D} = \varphi_{0} + \varphi_{1} + \cdots + \varphi_{\nu},
\]
where $\operatorname{supp} \varphi_{i} \subset U_{i}$ for each $i =
0,\ldots,\nu$.

\bigskip\noindent\emph{Step~$3$.} Given $f \in F_{0}(D) \cap W^{1,2}(D)$, we
claim that $\forall\, \varepsilon > 0$, $\exists\, h \in C_{c}^{\infty}(D)$
satisfying
\[ %\tag{5.29} \label{e5.29}
  \|f-h\|_{W^{1,2}(D)} = O(\varepsilon).
\]
By Lemma~A, we see that $\forall\, i \in \{ 1,\ldots,\nu \}$, there exists a
corresponding cut-off function $\eta_{i} \in C_{\Gamma_{i}}^{\infty}(U_{i})$
such that $\|f-h_{i}\|_{W^{1,2}(U_{i})} = O(\varepsilon)$, where $h_{i} := f
\cdot \eta_{i}$. Let $h_{0} := f \in C^{\infty}(D)$. Define
\begin{equation} \label{e3.19}
  h := \varphi_{0} h_{0} + \varphi_{1} h_{1} + \cdots + \varphi_{\nu} h_{\nu},
\end{equation}
then $h$ is smooth. Evidently the support $\operatorname{supp} (\varphi_{0}
h_{0}) \subset U_{0} \subset \subset D$, $h_{i} = f \cdot \eta_{i} \in
C_{\Gamma_{i}}^{\infty}(U_{i})$, and $\operatorname{supp} (\varphi_{i} h_{i})
\subset\subset D$, $\forall\, i = 1,\ldots,\nu$. Hence $\operatorname{supp} h
\subset\subset D$ and $h \in C_{c}^{\infty}(D)$. We can write
\[ %\tag{5.31} \label{e5.31}
  f-h = \varphi_{1}(f-h_{1}) + \cdots + \varphi_{\nu}(f-h_{\nu}),
\]
since $\varphi_{0}(f-h_{0}) \equiv 0$. It can be seen that
\begin{equation} \label{e3.20}
\begin{split}
  &\quad \|f-h\|_{W^{1,2}(D)}^{2} \\
  &\leq C \sum_{i=1}^{\nu} \big\{ \int_{U_{i}} \varphi_{i}^{2} |f-h_{i}|^{2}
    + \int_{U_{i}} \varphi_{i}^{2} |Df-Dh_{i}|^{2}
    + \int_{U_{i}} |D\varphi_{i}|^{2} |f-h_{i}|^{2} \big\} \\
  &\leq C \sum_{i=1}^{\nu} \|f-h\|_{W^{1,2}(U_{i})}
  = O(\varepsilon),
\end{split}
\end{equation}
since $\varphi_{i}^{2}$ and $|D\varphi_{i}|^{2}$ are zero outside $U_{i}$,
$\varphi_{i}^{2} \leq 1$ and $|D\varphi_{i}|^{2} < C$, independent of
$\varepsilon$. The last assertion \eqref{e3.20} is based on Lemma~A.

\bigskip\noindent\emph{Step~$4$.} It remains to show Theorem~Z for the case
with volume constraint, i.e., to show the second statement \eqref{e3.2}. Recall
the definition of $H(D)$ in \eqref{e2.1}. Given $f \in \mathcal{G}_{0}(D) \cap
W^{1,2}(D)$ we have to prove that $\forall\, \varepsilon > 0$, $\exists\,
\widetilde{h} \in \mathcal{G}_{c}(D)$ such that
\begin{equation} \label{e3.21}
  \|f-\widetilde{h}\|_{W^{1,2}(D)} = O(\varepsilon).
\end{equation}
Still consider $h_{i} := f \cdot \eta_{i}$ (as in Step~2) on each $U_{i}$, $i =
1,\ldots,\nu$, and $h_{0} := f$ on $D$. Define $\eta_{0} \in C^{\infty}_{c}(D)$
by $\eta_{0} \equiv 1$ on the open set $G_{0}$, defined in Step 2, and $0\leq
\eta_{0} \leq 1$ on $D$. Let $\eta_{i}$
and $\delta_{i}$ be defined as in the proof of Lemma~A, where $(U_{i},
\Gamma_{i}, V_{i})$ are considered, $i = 1,\ldots,\nu$. Define
\[ %\tag{5.34} \label{e5.34}
  \eta
  := \varphi_{0} \eta_{0} + \varphi_{1} \eta_{1} + \cdots + \varphi_{\nu}
  \eta_{\nu}.
\]
 Then $\eta \geq 0$ on $D$, and $\int_{D}{\eta} \gneqq 0$. Choose $\delta > 0$
 such that $\delta > 4\delta_{i}$, $\forall\, i = 1,\ldots,\nu$, and consider
 $D_{\delta} := D - N_{\delta} \subset\subset D$, where $N_{\delta}$ is the
 $\delta$-tubular neighborhood of $\partial D$ in $D$. Let $h$ be given by
 \eqref{e3.19} and define
\begin{equation} \label{e3.22}
  \widetilde{h} := h - \mu \eta, \quad \mu := \frac{\int_{D} h}{\int_{D} \eta}.
\end{equation}
Clearly, $h = \varphi_{0} f\eta_{0} + \varphi_{1} f\eta_{1} + \cdots +
\varphi_{\nu} f\eta_{\nu} = f\eta \in C_{c}^{\infty}(D)$, and $\int_{D}
\widetilde{h} = 0$ by \eqref{e3.22}. Hence $\widetilde{h} \in
\mathcal{G}_{c}(D)$. Now claim that $\mu$ is small. By $f \in
\mathcal{G}_{0}(D)$, we have $\int_{D} f = 0$,
\begin{align*} %\tag{5.36} \label{e5.36}
  \big| \int_{D} f\eta \, \big|
  &= \big| \int_{D_{\delta}} f + \int_{D-D_{\delta}} f\eta \, \big|
  = \big| \big( \int_{D} f - \int_{D-D_{\delta}} f \big)
    + \int_{D-D_{\delta}} f\eta \, \big| \\
  &= \big| \int_{D-D_{\delta}} f(1-\eta) \big|
  \leq \int_{D-D_{\delta}} |f|.
\end{align*}
Hence,
\[ %\tag{5.37} \label{e5.37}
  \big| \int_{D} f\eta \big|^{2}
  \leq \big( \int_{N_{\delta}} |f| \big)^{2}
  \leq \big( \int_{N_{\delta}} |f|^{2} \big) \cdot |N_{\delta}|
  \leq C \big( P(\delta) \cdot \delta^{2} \big) \cdot \delta
\]
by Lemma~B. Thus,
\begin{equation} \label{e3.23}
  |\mu|^{2}
  = \frac{\big| \int_{D} f\eta \big|^{2}}{\big| \int_{D} \eta \big|^{2}}
  = o(\delta^{3})
  \to 0
\end{equation}
as $\delta \to 0$. Finally, we obtain
\begin{align*}
%\tag{5.39} \label{e5.39}
\begin{split}
  \int_{D} |f-\widetilde{h}|^{2}
  &\leq C \big( \int_{D} |f-h|^{2} + \int_{D} |\mu \eta|^{2} \big) \\
  &\leq C \int_{D} |f-h|^{2} + o(\delta^{3}),
\end{split} \\
%\tag{5.40} \label{e5.40}
  \int_{D} |Df-D\widetilde{h}|^{2}
  &\leq C \big( \int_{D} |Df-Dh|^{2} + \int_{D} \mu^{2} |D\eta|^{2} \big).
\end{align*}
Thus,
\[ %\tag{5.41} \label{e5.41}
  \|f-\widetilde{h}\|_{W^{1,2}(D)}^{2}
  \leq C \big( \|f-h\|_{W^{1,2}(D)}^{2} + \mu^{2} \int_{D} |D\eta|^{2} \big).
\]
Using \eqref{e3.14} and \eqref{e3.23}, the last term is of the order
$o(\delta^{3} \cdot 1/\delta^{2}) = o(\delta)$, tending to zero as $\delta \to
0$. But in Step~3, we have shown \eqref{e3.20}. Thus, \eqref{e3.21} is proved.
The proof of Theorem~Z is completed.
\end{proof}

A second proof of Theorem~Z reduces to the case of smooth domains by shrinking
a given Lipschitz domain $D$ into $D' \subset\subset D$ where $S \equiv
\partial D'$ is smooth and lies in
\[ %\tag{5.42} \label{e5.42}
  N_{\delta}
  := \{ p \in D \:;\: \operatorname{dist}(p, \partial D) < \delta \}.
\]

The presentation of the second proof is to show clearly how the Lipschitz
condition is essential for the three closure theorem.  Recall that the ``focal
distance" $s(p)$ of $S$ at $p \in S$ means the distance from $p \in S$ to the
first focal point $\overline{p} = p + s(p) \cdot \nu$, at which the map $p
\mapsto p+t\nu$ is singular for the smallest $t > 0$. Here $\nu = \nu(p)$
denotes the unit normal of $S$ inward $D$ at $p$. \\

The idea of the second proof is sketched as follows. If we can control the
minimal focal distance $\overline{s} := \min \{ s(p) \:;\: p \in S \}$, so that
$\overline{s} > c\delta$, $c$ being a constant independent of $\delta$, as
$\delta$ tends to zero, then we can construct $\widetilde{\eta} \in
C_{c}^{\infty}(D)$ with $|D\widetilde{\eta}|^{2}$ controllable, where
$\widetilde{\eta} \equiv 0$ for $t \leq 0$; $0 \leq \widetilde{\eta} \leq 1$
for $0\leq t \leq c\delta < \overline{s}$, and $\widetilde{\eta} \equiv 1$
elsewhere on $D$. Given $f \in \mathcal{F}(D) \cap W^{1,2}(D)$ and any
$\varepsilon > 0$. By letting $h := f \cdot \widetilde{\eta}$, we can select
$\delta > 0$ sufficiently small and show that $\|f-h\|_{W^{1,2}(D)} =
O(\varepsilon)$.\\

\begin{wrapfigure}[10]{r}{0.35\textwidth}
\centering
\includegraphics[width=0.15\textwidth]{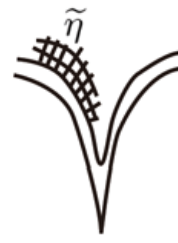}
\caption{focal distance $\overline{s}$ is as small as $o(\delta)$} \label{F12}
\end{wrapfigure}
Again, the Lipschitz condition of $\partial D$ is essential to control $|D
\widetilde{\eta}|^{2}$. For example, let $n = 2$ and consider $\partial D$
locally as a graph of $r = u(x) = |x|^{\alpha}$, $0 < \alpha < 1$, with a cusp
at $x = 0$. The minimal focal distance is smaller than the order of $\delta$,
i.e., $\overline{s} = o(\delta)$ (see Figure~\ref{F12}). One can show that
$|D\widetilde{\eta}|^{2} \neq O(1/\delta^{2})$ and $|D\widetilde{\eta}|^{2}
\cdot \delta^{2} \to \infty$ as $\delta \to 0$. Thus $h \equiv f \cdot
\widetilde{\eta}$ in this example does not tend to $f$ in $W^{1,2}(D)$ when
$\delta \to 0$.

\begin{lemC} \label{LemC}
Given a Lipschitz simple domain $(U,\Gamma,V)$ and a small number $\delta > 0$,
let $r = \overline{u}(x) = u(x) + 2\delta$, $x \in V$, as in Step~$1$ of
Lemma~\textup{A}. Consider the smooth approximation $w_{\alpha}(x)$ of
$\overline{u}(x)$ defined by \eqref{e3.6}. Let $n = 2$. By choosing $\alpha =
\delta/L$, we have
\begin{enumerate}
\item[\textup{(i)}] $|w_{\alpha}(x)-\overline{u}(x)| \leq L \cdot \alpha =
\delta$ and $|w''_{\alpha}(x)| \leq 2C_{2} \cdot L/\alpha = 2C_{2}
L^{2}/\delta$, where $C_{2} = \max \{ |\varphi''(x)| \:;\: -1 \leq x \leq 1
\}$, $L$ is the Lipschitz constant and $\varphi$ is the standard modifier.

\item[\textup{(ii)}] the minimal focal distance $\overline{s} \geq c\delta$
by choosing $\alpha = \delta/L$ and $c\equiv 1/(2C_{2} L^{2})$.
\end{enumerate}
\end{lemC}
\begin{proof}
The computations are straightforward. The first inequality of (i) has been
shown in \eqref{e3.7}, letting $\alpha = \delta/L$. The second can be
calculated directly, except minding the formula:
\[ %\tag{5.43} \label{e5.43}
  \int_{V} \overline{u}(x) \varphi'' \big( \frac{y-x}{\alpha} \big) \, dy
  = \alpha \overline{u}(x) \big[ \varphi' \big( \frac{y-x}{\alpha} \big)
    \big]_{x-\alpha}^{x+\alpha}
  = 0.
\]
As for (ii), we see that
\[ %\tag{5.44} \label{e5.44}
  |K(x)|
  = \big| w''_{\alpha}(x)/\! \sqrt{1+w'_{\alpha}(x)^{2}} \big|
  \leq |w''_{\alpha}(x)|
  \leq 2C_{2} L/\alpha
  = 2C_{2} L^{2}/\delta,
\]
where $K(x)$ is the curvature of $r = w_{\alpha}(x)$ at $x$. The radius $R(x)$
of the curvature circle at $x$ is $1/K \geq \alpha/(2C_{2}L) = C\delta$,
$\forall\, x \in V$, by denoting $c := 1/(2C_{2} L^{2})$. Thus,
\[ %\tag{5.45} \label{e5.45}
  \overline{s}
  = \inf \{ R(x) \:;\: x \in V \}
  \geq c\;\delta. \qedhere
\]
\end{proof}

Lemma~C is also valid for general dimension $n$. The proof is similar, yet a
more tedious language is needed to describe the aforementioned idea for $n =
2$.\\

To finish the second proof of Theorem~Z, we have to start with Lipschitz simple
domains $(U_{i},\Gamma_{i},V_{i})$ locally and construct $\widetilde{\eta}_{i}
\in C_{\Gamma_i}^{\infty}(U_{i})$ for $i = 1,\ldots,\nu$. Consider the smooth
surface $S_{i}$ defined by $r = w_{i}(x)$ approximating $r = \overline{u}(x)$
as before. Denote it's minimal focal distance by $\overline{s}_{i}$. By
Lemma~C, $\exists\, c_{0}$ with $0 < c_{0} < 1$ such that $\overline{s}_{i} >
2c_{0} \delta$, $\forall\, i = 1,\ldots,\nu$. Define $\widetilde{\eta}_{i}$ by
\[ %\tag{5.46} \label{e5.46}
  \widetilde{\eta}_{i}(p+t\nu)
  := \widetilde{\eta}_{0} \big( \frac{t}{c_{0} \delta} \big),
\]
where $\widetilde{\eta}_{0} \in C^{\infty}(\mathbb{R}^{1})$ is a cut-off
function on $\mathbb{R}^{1}$ such that $\widetilde{\eta}_{0}(x) = 0$ for $x
\leq 0$; $0 \leq \widetilde{\eta}_{0}(x) \leq 1$ for $0 \leq x \leq 1$; and
$\widetilde{\eta}_{0}(x) \equiv 1$ for $x \geq 1$. Let $h_{i} := f
\widetilde{\eta}_{i}$, $\forall\, i = 1,\ldots,\nu$. Introduce the partition of
unity $\{ \varphi_{0}, \varphi_{1}, \ldots, \varphi_{\nu} \}$ as in the first
proof and define $h$ by \eqref{e3.19}. We have $h \in C_{c}^{\infty}(D)$.
Clearly, $|D\widetilde{\eta}_{i}|^{2} = |D_{t} \widetilde{\eta}_{i}|^{2} \leq
\frac{\widetilde{C}_{1}}{c_{0}^{2} \delta^{2}}$, where $\widetilde{C}_{1}$ is
the bound of the term $|\widetilde{\eta}'_{0}|^{2}$, a constant independent of
$\delta$. The rest of the proof, including the treatment of the case with
volume constraint, is essentially the same with the first proof.\\

In summary, the two proofs share the following key points:
\begin{enumerate}
\item[(i)] the estimations of $|D\eta|^{2}$ and $|D\widetilde{\eta}|^{2}$,
which rely basically on the Lipschitz condition on the boundary $\partial D$,
are provided by Lemmas~A and C,

\item[(ii)] when the domain is reduced to $n = 1$, where Lipschitz condition
makes no sense, the essential part of the proofs is the estimation of
$\int_{N_{\delta}} f^{2}$, based on Lemma~B.
\end{enumerate}
One might consider a third proof by flattening the Lipschitz boundary $\partial
D$ locally, which is not attempted in this paper.

%%========================================================= sec. 4 ==
\section{Proof of Sobolev continuity} \label{S4}
%%===================================================================
%%=========================================== sec. 3.1 -> sec. 4.1 ==
\subsection{Proof of Theorem~S} \label{S4.1}
%%===================================================================
\begin{proof}[Proof of Theorem~\textup{S}]
\emph{Step~$1$.} Claim that $H_{t} = \bigcap_{r > t} H_{r}$: ``$\subset$" is
obvious, since $H_{t} \subset H_{r}$, $\forall\, r > t$. We shall show the
converse ``$\supset$" by using $\overline{D(t)} = \bigcap_{r > t}
\overline{D(r)}$ as follows. Given $g \in \bigcap_{r > t} H_{r}$, it is not
difficult to see that $\operatorname{supp} g \subset D(t)$:

\medskip\noindent (i) For an open set $V \subset M^{n}$ with $\overline{V} \cap
\overline{D(t)} = \phi$, there exists $r' > t$ such that $\overline{V} \cap
D(r') = \phi$. In fact, suppose the contrary, there exist $r_{n}$ and $x_{n}$,
with $r_{n} > t$, tending to $t$, and $x_{n} \in \overline{V} \cap D(r_{n})$.
Consider a convergent subsequence, still denoted by $x_{n}$, which tends to
$x_{0} \in \overline{V}$. Given $r > t$, $\exists\, N$ such that $r > r_{n} >
t$, $\forall\, n > N$. Then $x_{n} \in D(r_{n}) \subset D(r)$ and $x_{0} =
\lim_{n} x_{n} \in \overline{D(r)}$. Thus $x_{0} \in \bigcap_{r > t}
\overline{D(r)} = \overline{D(t)}$, by \eqref{e2.4}. This contradicts the
assumption $\overline{V} \cap \overline{D(t)} = \phi$. Remark that the previous
argument \emph{does not} apply to the counterexample in Figure~\ref{F5}, where
$D(t)$  is continuous \emph{in measure}, yet the condition $\bigcap_{r > t}
\overline{D(r)} = \overline{D(t)}$ is not satisfied.

\medskip\noindent (ii) Suppose that there exists an open set $V$ in $M^{n}$
with $\overline{V} \cap \overline{D(t)} = \phi$ and $g \not\equiv 0$ a.e.~on
$V$. By (i), $\exists\, r' > t$ with $\overline{V} \subset M^{n}-D(r')$. Hence
$g \notin H_{r'}$, against the assumption that $g \in \bigcap_{r > t} H_{r}$.
Therefore, $\operatorname{supp} g \subset \overline{D(t)}$. By \eqref{e2.6},
the interior $(\overline{D(t)})^{0} = \overline{D(t)} -
\partial(\overline{D(t)}) = \overline{D(t)} - \partial(D(t)) = D(t)$. We see
that supp $g \subset D(t)$ and $g \in L^{2}(D(t))$. Remark that for the
monotone continuum $\mathcal{D'}_{4}$ in Example~\ref{E3.3}, the last statement
supp $g \subset D(t)$ for $t = 1$ is false, as $\mathcal{D'}_{4}$ is not
set-continuous, though it is continuous in Hausdorff measure.

\medskip\noindent (iii) Clearly, $\int_{D(t)} g = 0$, since $g \in H_{r}$,
$\forall\, r > t$, and $\operatorname{supp} g \subset D(t)$.\\

\bigskip\noindent\emph{Step~$2$.} Obviously, $g$ has its weak derivatives
$h_{i} := D_{i} g \in L^{2}(D(t))$, $\forall\, i = 1,\ldots,n$. In fact, by $g
\in H_{r}$, $\forall\, r > t$, there are $h_{i} \in L^{2}(D(r))$ such that
\begin{equation} \label{e4.1}
  \int_{D(r)} g \cdot D_{i} \varphi = -\int_{D(r)} h_{i} \cdot \varphi,
    \quad \forall\, \varphi \in C^{\infty}_{c}(D(r)).
\end{equation}
We shall see that $h_{i} \big|_{D(r) - \overline{D(t)}} = 0$ almost everywhere.
If there exists an open set $U \subset D(r) - \overline{D(t)}$ such that $h_{i}
\not\equiv 0$ a.e.~on $U$. Choose $\varphi \in C^{\infty}_{c}(D(r))$ with
$\overline{\operatorname{supp} \varphi} \subset U$ and $\int_{U} h_{i} \cdot
\varphi \neq 0$. This violates \eqref{e4.1}, since its LHS $= 0$ by Step~1, but
the RHS $\neq 0$. Hence, $h_{i} \in L^{2}(D(t))$, $i=1,2,\ldots,n$, are the
weak derivatives of $g$ in $D(t)$.\\

\bigskip\noindent\emph{Step~$3$.} We now come to the subtle part in the proof
of Theorem~S. To claim $g \in H_{t}$, it remains to show that the trace $Tg$ of
$g$ is zero on $\partial(D(t))$. In fact, we will show that $g$ lies in the
closure of $\mathcal{G}_{c}(D(t))$ in $W^{1,2}(D(t))$.

\medskip\noindent (i) Given $g \in \bigcap_{r > t} H_{r}$. By using Theorem~Z,
there exist $f_{n} \in \mathcal{G}_{c}(D(r_{n}))$, $r_{n} \searrow t$, such
that $\|f_{n}-g\|_{H_b}^{2} < 1/n$, tending to $0$ as $n \to \infty$. Now claim
that $\lim_{n} f_{n}|_{\partial(\overline{D(t)})} = 0$. There are two parts
$\Lambda$ and $\Lambda'$ of $\partial(\overline{D(t)})$, which we have to deal
with separately. The part $\Lambda$ basically consists of the points on
$\partial \overline{D(t)}$, where the domain $D(t)$ extends to a real ``larger"
$D(r)$ with $r > t$;  and $\Lambda'$ is the part of $\partial \overline{D(t)}$
on which $D(t)$ ceases to enlarge, i.e., $\exists \, r_{0} > t$  such that
$\partial D(r)$ is identical with $\partial D(t)$ for any $r$,  $t<r<r_{0}$.
More precisely, writing $U \equiv D(b) - \overline{D(t)}$, we consider $\Lambda
\equiv \partial \overline{D(t)} \cap \overline{U}$, and $\Lambda' \equiv
\partial \overline{D(t)} - \Lambda$. Then $\partial D(t)=
\partial\overline{D(t)} = \Lambda \cup \Lambda'$, noting that the joint set of
$\partial D(t)$ contains \emph{no open set of} any $\Gamma_{i}$, which is in a triple $(U_i,\Gamma_i,V_i)$ given in Definition~\ref{D2.2}. (see Remark~\ref{R2.8}). We will show that $\lim_{n} f_{n} = 0$ a.e.~on both $\Lambda$ and $\Lambda'$, which implies that $f_{n} \to 0$ almost everywhere on $\partial D(t)$ as $n \to \infty$.\\

\medskip\noindent (ii) Let $U = D(b) - \overline{D(t)}$. We have
$\|f_{n}\|_{W^{1,2}(U)}^{2} = \|f_{n}-g\|_{W^{1,2}(U)}^{2} \leq
\|f_{n}-g\|_{H_b}^{2} \to 0$ as $n \to \infty$. where $f_{n} \in
\mathcal{G}_{c}{(D(r_{n}))}$, $r_{n} \searrow t$. Remark that in (ii) of
Step~1, we have shown that $g = 0$ almost everywhere on $U$. Now we claim that
$f_{n} \to 0$, a.e. on $\Lambda \equiv \partial \overline{D(t)} \cap
\overline{U}$. Given $p \in \Lambda$, \emph{not} a joint point of $\partial
D(t) = \partial \overline{D(t)}$. As $\Lambda$ is a Lipschitz boundary around
$p$, there is a Lipschitz simple  domain $(U_{p},\Gamma_{p},V_{p})$ in $U$,
where $U_{p} \subset U$ and $p \in \Gamma_{p} \subset \Lambda$. Consider
correspondingly the local coordinate $(x,r)$ of $U_{p}$, such that the graph $r
= u(x)$ is $\Gamma_{p}$ , in which $x \in V_{p} \subset \mathbb{R}^{n-1}$ and
$u(x)$ is a Lipschitz function of $x$. Let $U_{p} = \{ (x,r) ; u(x) < r < v(x)
\}$. Choose $n >$ some $n_{0}$, so that $W \equiv \{ v(x)-a < r < v(x) \}$ is
disjoint from $\overline{D(r_{n})}$, where $a$ is a small number. This is
guaranteed by (i) of Step~1. Given $n > n_{0}$. Since $f_{n} \in
\mathcal{G}_{c}(D(r_{n}))$, we may regard $f_{n} \in \mathcal{G}_{c}(D(b))$ by
extending $f_{n}$ on $D(r_{n})$ to the whole $D(b)$ with $f_{n}$ vanishing
outside $D(r_{n})$. We see that $f_{n} = 0$ on $W$, and have
\begin{equation} \label{e4.2}
\begin{split}
  \int_{\Gamma_{p}} f_{n}^{2} \, dx
  &= \int_{\Gamma_{p}} \left( -\int_{u(x)}^{v(x)}
    \frac{\partial f_{n}}{\partial r} \, dr  \right)^{2} \, dx
  \leq C\int_{\Gamma_{p}} \int_{u(x)}^{v(x)}
    \left( \frac{\partial f_{n}}{\partial r} \right)^2 \, dr dx \\
  &\leq C\int_{U_{p}} |Df_{n}|^{2}
  \leq C\|f_{n}\|^{2}_{W^{1,2}(U)} \to 0 \quad \textrm{as $n \to \infty$}.
\end{split}
\end{equation}
By \eqref{e4.2}, $f_{n} \to 0$ on $\Gamma_{p}$ as $n \to \infty$. This proves
the claim, since the joint set is of measure zero. Remark that at a joint point
$p \in \Lambda$, the triple $(U_{p},\Gamma_{p},V_{p})$ with $U_{p} \subset U$
may not exist, since $\Lambda$ is not a Lipschitz boundary of $U$. For example,
look at $-p_0$ in $\mathcal{D}_{3}$ of Example~\ref{E3.2} (see
Figure~\ref{F6}). \\

\medskip\noindent (iii) Let $\Lambda' \equiv \partial \overline{D(t)} -
\Lambda$. We see that $\Lambda'$ is the subset of $\partial \overline{D(t)}$ on
which $\partial \overline{D(t)} = \partial \overline{D(r)}$ for any $r$, with
$t<r<b$. In fact, any point $q \in \Lambda'$ has a neighborhood $Q$, disjoint
from $U$. As $U= D(b)-\overline{D(t)} = \bigcup_{r > t}
(D(r)-\overline{D(t)})$, we have $Q \cap (D(r)-\overline{D(t)}) = \phi$, for
any $r$ with $t<r<b$. Hence in $Q$, $\overline{D(r_{n})} = \overline{D(t)}$,
$\forall\, r_{n} > t$. Therefore, on $\Lambda'$, it is clear that $\partial
\overline{D(r_{n})} = \partial \overline{D(t)}$, $\forall\, r_{n} > t$.
However, $f_{n} \in \mathcal{G}_{c}(D(r_{n}))$. We see that $f_{n}
\big|_{\Lambda'} = 0$ on $\partial \overline{D(r_{n})} = \partial
\overline{D(t)}$, $\forall\, r_{n} > t$, and therefore, $f_{n} \equiv 0$ on
$\Lambda'$. \\

\medskip\noindent (iv)
We have obtained $\lim_{n} f_{n} = 0$ a.e.~on the whole $\partial D(t)$, as
$\Lambda \cup \Lambda' = \partial \overline{D(t)} = \partial D(t)$ by
\eqref{e2.5}. Remark that in the example $\mathcal{D}_4$, $\partial
\overline{D(t)} \subsetneqq \partial D(t)$, since $L_{0} \subset \partial D(t)$
but not included in $\partial\overline{D(t)}$. The last argument is not valid,
which says that the Sobolev continuity may not be true for $\mathcal{D}_4$.  \\

\medskip\noindent (v) Now use the arguments for proving Theorem~Z to conclude
that the trace $Tg \equiv 0$ on $\partial D(t)$. Apply the proof of Lemma~A
locally to $f_{n}$ on each Lipschitz simple domain $(U_{i},\Gamma_{i},V_{i})$
given by Definition~\ref{D2.2}. Note that the assumption of  $f|_{\Gamma_{i}} =
0$ in Lemma~A is now replaced by $\int_{\Gamma_{i}} f_{n}^{2} \to 0$ as $n \to
0$, and hence $\lim_{n} f_{n} = 0$  almost everywhere on $\Gamma_{i}$ for each
$n$. The arguments to show \eqref{e3.4} in Lemma A are still valid, except that \eqref{e3.16}--\eqref{e3.18} of Lemma~B should be modified. Remark that
$P(\delta):= \int_{N_{\delta}} |{Df}_{n}|^{2}$ still tends to 0 as $\delta \to
0$, since $\int_{N_{\delta}} |{Df}_{n}|^{2} \leq C \int_{N_{\delta}} |Dg|^{2}$
which is independent of $n$, and $\int_{N_{\delta}} |Dg|^{2}$ is of order
$O(\delta)$ as $g\in W^{1,2}(D)$. The estimation of \eqref{e3.16} becomes
\begin{equation} \label{e4.3}
\begin{split}
  &\quad \int_{N_{\delta}} f_{n}^{2} \, dr dx \\
  &= \int_{V} \int_{0}^{\delta} \left\{ f_{n}(x,0)^{2}
    + 2f_{n}(x,0) \int_{0}^{r} f'_{n}(x,s) \, ds
    + \left( \int_{0}^{r} f'_{n}(x,s) \, ds \right)^{2} \right\} \, dr dx \\
  &\leq \left( \int_{\Gamma} f_{n}^{2} \right) \cdot O(\delta)
        + \left( \int_{\Gamma} f_{n}^{2} \right)^{1/2} \cdot O(\delta^{3/2})
    + \int_{N_{\delta}} |Dg|^{2} \cdot O(\delta^{2}),
\end{split}
\end{equation}
respectively. The first two terms tend to $0$ as $n \to 0$. Thus for $f_{n}$,
Lemma~B as well as Lemma~A are still valid, where we consider $h_{n} = f_{n}
\cdot \eta \in C^{\infty}_{\Gamma}(U)$ in \eqref{e3.4}. \\

\medskip\noindent (vi)
The difficulty here is that we know only $\lim_{n} f_{n} = 0$, almost
everywhere on $\partial D(t)$. The function $f_{n}$ may approximate to a jump
function, which is nonzero at an isolated joint point $p_{0} \in \partial D(t)$
and vanishes elsewhere. In that case, $g$ may not have trace $Tg = 0$ around
$p_{0}$. Fortunately, the estimates in Lemma~A and Lemma~B can be modified as
the above \eqref{e4.3}. Thus the crucial inequalities \eqref{e3.11},
\eqref{e3.15} and \eqref{e3.8} are still valid. Again, by the partition of
unity, there exists $h_{n} \in C^{\infty}_{c}(D(t))$ (still denoted by $h_{n}$
for convenience), such that $\|g-h_{n}\|_{W^{1,2}(D(t))}^{2} \to 0$, as $n \to
\infty$. This is equivalent to say that the trace $Tg = 0$ on $\partial D(t)$,
i.e., the ``boundary value" of $g$ is zero on $\partial D(t)$. By Theorem~Z,
it also means that $g \in W_{0}^{1,2}(D(t)) = E(D(t))$.

\medskip\noindent (vii)
In general, given $g \in E(D)$, it is evident that $g \in H(D)$ iff $\int_{D} g
= 0$.  However, we have proved that $\int_{D(t)} g = 0$ in Step~1, and $g \in
E(D(t))$ by (vi), it is concluded that $g \in H_{t}$. The second equality of
the Sobolev continuity in Theorem~S is proved. \\

\bigskip\noindent\emph{Step~$4$.} Now we show that $\overline{\bigcup_{s < t}
H_{s}} = H_{t}$, the first equality in Theorem~S. To claim ``$\subset$": Given
$f \in \overline{\bigcup_{s < t} H_{s}}$. Note that the closure is taken in
$H_{b}$, where we regard $H_{s} \subset H_{b}$, $\forall\, s \leq b$, as in the
previous manner. Evidently, $\operatorname{supp} f \subset D(t)$. Suppose on
the contrary, $\exists$ open set $U \subset D(b) - \overline{D(t)}$, such that
$\int_{U} |f|^{2} = c > 0$. By the assumption of $f$, $\exists\, s_{0} < t$ and
$h_{0} \in H_{s_{0}}$ with $\|f-h_{0}\|^{2}_{H_{b}} < c/2$, which contradicts
$\|f-h_{0}\|^{2}_{H_{b}} \geq \int_{U} |f|^{2} = c$. Thus $f \in L^{2}(D(t))$.

\bigskip\noindent\emph{Step~$5$.} Given $\epsilon > 0$, there exist
$\overline{s} < t$ and $g \in H_{\overline{s}}$ such that $\|f-g\|_{H_{b}} <
\epsilon/2$. By Theorem~Z, $H_{\overline{s}}$ is the closure of
$\mathcal{G}_{c}(D(\overline{s}))$ in $W^{1,2}(D(\overline{s}))$. There is $h
\in \mathcal{G}_{c}(D(\overline{s}))$ with $\|g-h\|_{H_{b}} < \epsilon/2$,
noting that $\|\cdot\|_{H_{s}} = \|\cdot\|_{H_{b}}$, $\forall\, s \leq b$.
Hence $\|f-h\|_{H_{b}} \leq \epsilon$, where $h \in
\mathcal{G}_{c}(D(\overline{s})) \subset \mathcal{G}(D(t))$, which means that
$f$ is in the closure of $\mathcal{G}(D(t))$ in $W^{1,2}(D(t)) = H_{t}$.
Therefore, $\overline{\bigcup_{s < t} H_{s}} \subset H_{t}$. Remark that
$\mathcal{G}_{c}(D(\overline{s})) \subset \mathcal{G}(D(t))$, by letting the
values of Sobolev functions in $\mathcal{G}_{c}(D(\overline{s}))$  be zero
outside $D(\overline{s})$.

\bigskip\noindent\emph{Step~$6$.} It remains to show ``$\supset$". We need to
use $\bigcup_{s < t} D(s) = D(t)$. Given $f \in H_{t}$ and $\varepsilon > 0$,
there exists $g \in \mathcal{G}_{c}(D(t))$ with $\|f-g\|_{H(t)} < \varepsilon$,
by Theorem~Z. Let $W$ be an open set with $\overline{W} \subset D(t)$, we claim
that $\exists\, s' < t$ such that $\overline{W} \subset D(s')$. Suppose on the
contrary, then $\forall\, s_{n} \nearrow t$, $\exists\, x_{n} \in D(s_{n})^{c}
\cap \overline{W}$, where $D(s)^{c}$ denotes the complement of $D(s)$ in
$\overline{D(b)}$. Select a convergent subsequence of $\{x_{n}\}$, still
denoted by $\{x_{n}\}$. Let $x_{0} = \lim_{n} x_{n}$. Given any $s < t$, there
exist $x_{n} \in D(s_{n})^{c} \subset D(s)^{c}$ for $n$ big enough. Hence
$x_{0} \in D(s)^{c}$, since $D(s)^{c}$ is closed in $\overline{D(b)}$, as
$D(b)$ is also assumed to be relatively compact. Hence, we have that $x_{0} \in
\bigcap_{s < t} D(s)^{c} = \big( \bigcup_{s < t} D(s) \big)^{c} = D(t)^{c}$,
contradicting $\overline{W} \subset D(t)$. As $g \in \mathcal{G}_{c}(D(t))$,
$\overline{\operatorname{supp} g} \subset D(t)$. There exists $s' < t$, such
that $\overline{\operatorname{supp} g} \subset D(s')$. Hence $g \in
\mathcal{G}_{c}(D(s')) \subset H_{s'}$, which yields that $f \in
\overline{\bigcup_{s < t} H(s)}$. The proof of Theorem~S is completed.
\end{proof}

\begin{rmk} \label{R4.1}
The separation properties of $V$ and $W$ stated in Step~1 and Step~6 are not
true in general for a monotone continuum of domains in $M^{n}$, It is
interesting to look at $V$ and $W$ in Figure~\ref{F5} of \S\ref{S3.1}, where a
monotone continuum is continuous in measure.
\end{rmk}

We have said in \S\ref{S2.1} that the set-continuity implies the Hausdorff
continuity. Basically, a similar argument of Step~1 in the proof of Theorem~S
can be used to show the implication. However, the converse is not true, as
demonstrated by Examples $\mathcal{D}_{4}$, $\mathcal{D'}_{4}$,
$\mathcal{D}_{5}$ and $\mathcal{D'}_{5}$ of \S\ref{S3.2}. Furthermore,
Theorem~S is not valid for the monotone family of Example~\ref{E3.2}, which is
exactly one-dimensional case of $\mathcal{D}_4$.\\

Few more words to comment about the subtle arguments in Step~3 of the proof of
Theorem~S: Given $g \in \bigcap_{r > t} H_{r}$, $g$ provides no ``sky-bridge"
passing through a joint point to connect the given simple Lipschitz domains.
This is an interesting intuitive observation supporting the proof of the trace
$Tg=0$. In the proof, we should be very careful concerning the behavior of $g
\in \bigcap_{r > t} H_{r}$ at the joint points of the ``generalized" Lipschitz
domain $D(t)$, as the joint points make the topological types varies. Some
complicated arguments are focused on Step~3 to prove that the trace $Tg$ vanish
on $\partial D(t)$. We use Figure~\ref{F6} in \S\ref{S3.1} to illustrate some
typical situation for $g$ that we have to avoid. The dangerous situation to be avoided is
that $g$ happens to be  $\chi_{D(t)}$
or its alike, around the joint point $-p_{0}$, at which the closure
$\overline{D(t)}$ glued together. (Here $\chi_{D(t)}$ denotes the
characteristic function of $D(t)$.) In that case, $g$ would not belong to
$H_{t}$, since $Tg$ would not vanish a.e. on $\partial{D(t)}$. Fortunately,
$\chi_{D(t)}$ is not weakly differentiable in $D(r), r>t$. Thus $\chi_{D(t)}$
is not in $H_{r}$ for any $r > t$. It says that the given Sobolev function $g$
in $\bigcap_{r > t} H_{r}$ can not be $\chi_{D(t)}$, or its alike. Step~3 of
the proof of Theorem~S, especially (v) and (vi) of it,  is elaborated to
exclude this kind of situation in general. The above observation that
$\chi_{D(t)}$ is not in $H_{r}$ for any $r > t$, i.e., there is no sky-bridge
connecting the upper and the lower parts of $D(t)$ at $-p_{0}$,  provides an
intuitive background underlying the arguments of the proof of the Sobolev
continuity. This marks a highlight in the proof of the Sobolev continuity.

%%=========================================== sec. 3.2 -> sec. 4.2 ==
\subsection{Proof of Theorem~B} \label{S4.2}
%%===================================================================
It remains to show Theorem~B concerning the continuity of eigenvalues of
the stability operator, as well as their strict monotonicity. Based on the Sobolev
continuity, proved in Theorem~S, the arguments \cite{FT90} by Frid--Thayer for
smooth continuums can be carried over to our $C^{0}$-continuums of generalized
Lipschitz domains. In order to provide an easier reference, we rewrite it in a
formulation consistent with this paper.\\

We shall only prove Theorem~B for $\widetilde{\lambda}_{k}$. The arguments also
work for $\lambda_{k}$, which is of a simpler case.

\begin{proof}[Proof of Theorem~\textup{B}]
\bigskip\noindent\emph{Step~$1$.}  Given $r_{i} \searrow t$, i.e., $\lim_{i \to
\infty} r_{i} = t$, $t < r_{i} < r_{i-1}$, $\forall\, i$. We claim the right
continuity of $\widetilde{\lambda}_{k}$. Denote by $\{u_{k}(t)\}$ an
orthonormal basis of eigenfunctions of $\widetilde{L}_{t} \equiv
\widetilde{L}(D(t))$ in $H_{t}$, corresponding to
$\{\widetilde{\lambda}_{k}(t)\}$. For each $r_{i}$, let
\[
  v_{1}(r_{i}), v_{2}(r_{i}), \ldots, v_{k}(r_{i}), \ldots
  \in H_{r_{i}}
  \subset H_{b}
\]
be an orthonormal base of eigenfunctions of $\widetilde{L}_{r_{i}}$
corresponding to $\widetilde{\lambda}_{j}(r_{i})$, $j = 1,2,\ldots,k,\ldots$.
For each $j$, $\widetilde{\lambda}_{j}(r_{i})$ is increasing in $i$ and bounded
by $\widetilde{\lambda}_{j}(t)$. Clearly, $\widetilde{\lambda}_{j}(r_{i}) \to$
some $\alpha_{j}$ as $i \to \infty$. We will show that $\alpha_{j} =
\widetilde{\lambda}_{j}(t)$.

\bigskip\noindent\emph{Step~$2$.} (Right-continuity) Fix $j \in \mathbb{N}$.
The set $\{ v_{j}(r_{1}), v_{j}(r_{2}), \ldots \}$ is bounded in $L^{2}(D(b))$,
since $\|v_{j}(r_{i})\|_{L^{2}} = 1$. Claim that it has a subsequence
converging to $v_{j}$ in $H_{t}$. We write $v(r_{i}) \equiv v_{j}(r_{i})$ at
the moment by dropping the index $``j"$ to simplify notations. In fact,
\begin{align*} %\tag{6.1} \label{e6.1}
  I_{D(b)}(v(r_{i}),v(r_{i}))
  &= \int_{D_{b}} |Dv(r_{i})|^{2} - |B|^{2} (v(r_{i}))^{2} \\
  &= \widetilde{\lambda}_{j}(r_{i}) \: \|v(r_{i})\|^{2}_{L^{2}}
  \leq \widetilde{\lambda}_{k}(t) \quad \textrm{(by \eqref{e2.15})},
\end{align*}
i.e., $\|Dv(r_{i})\|^{2} < C$, independent of $i$. Thus we have a subsequence
of $\{v(r_{i})\}$ which converges to $v_{j}$ in $L^{2}(D(b))$. Still denote the
subsequence by $v(r_{i})$. Let $w_{il} \equiv v(r_{i}) - v(r_{l})$, $i \geq l$.
Remark that $v(r_{i})$ is regular on $D(r_{i})$. Computing $\|Dw_{il}\|^{2}$ in
$L^{2}(D(b))$ and using the integration by parts, it is not difficult to show
that $\|Dw_{il}\|^{2} \to 0$, as $i,l \to \infty$. Thus $D_{h} v(r_{i})$ is a
Cauchy sequence in $L^{2}(D(b))$, where $D_{h}v(r_{i})$ denotes the weak
derivative along $h$ direction. Let $D_{h} v(r_{i}) \to \zeta_{h}$ in
$L^{2}(D(b))$. Then for each $\varphi \in C^{\infty}_{c}(D(b))$,
\[ %\tag{6.2} \label{e6.2}
  \langle \zeta_{h}, \varphi \rangle
  = \lim_{i} \langle D_{h} v(r_{i}), \varphi \rangle
  = \lim_{i} \langle v(r_{i}), D_{h} \varphi \rangle
  = \langle v_{j}, D_{h} \varphi \rangle,
\]
which shows that $\zeta_{h} = D_{h}v_{j}$ and $v_j \in H_b$.

\bigskip\noindent\emph{Step~$3$.} Claim that $v_{j}$ is the eigenfunction of
$\widetilde{L}_{t}$ with the eigenvalue $\alpha_{j}$. For each $l > 0$,
$v(r_{l}), v(r_{l+1}), \ldots \in H_{r_{l}}$. The limit $v_{j}$ is therefore in
$H_{r_{l}}$ for any $r_{i}$. By Theorem~S, $v_{j} \in \bigcap_{l=1}^{\infty}
H_{r_{l}} = \bigcap_{r=1}^{\infty} H_{r} = H_{t}$, since $H_{s} \subset H_{r}$
for $s \leq r$. Let $H^{2}_{t}$ be the closure of $\mathcal{G}(D(t))$ in
$W^{2,2}(D(t))$. For any $w \in H^{2}_{t}$, we see that
\begin{align*} %\tag{6.3} \label{e6.3}
  \langle v_{j}, \widetilde{L}_{t}w \rangle
  &= \lim_{i} \langle v(r_{i}), \widetilde{L}_{t}w \rangle
  = \lim_{i} \widetilde{I}_{t}(v(r_{i}),w) \\
  &= \lim_{i} \widetilde{\lambda}_{j}(r_{i}) \langle v(r_{i}), w \rangle
  = \alpha_{j} \langle v_{j}, w \rangle,
\end{align*}
which means that $I_{t}(v_{j},\varphi) = \alpha_{j} \langle v_{j}, \varphi
\rangle$, $\forall\, \varphi \in H_{t}$, since $H^{2}_{t}$ is dense in $H_{t}$.
Therefore,
\[ %\tag{6.4} \label{e6.4}
  \widetilde{L}_{t} v_{j} = \alpha_{j} v_{j},
\]
$\alpha_{j}$ being an eigenvalue of $\widetilde{L}_{t}$ on $H_{t}$.

\bigskip\noindent\emph{Step~$4$.} Now claim that $\alpha_{j}$ is the
\emph{$j$-th} eigenvalue of $\widetilde{L}_{t}$ on $H_{t}$. Clearly,
$\widetilde{\lambda}_{j}(t) \geq \lim_{i} \widetilde{\lambda}_{j}(r_{i}) =
\alpha_{j}$. Suppose $\widetilde{\lambda}_{1}(t) > \alpha_{1}$. As $\alpha_{1}$
is an eigenvalue of $\widetilde{L}_{t}$, $\alpha_{1}$ must be some
$\widetilde{\lambda}_{h}(t)$ with $h > 1$. Thus we have
\[
  \widetilde{\lambda}_{1}(t)
  \leq \widetilde{\lambda}_{h}(t)
  = \alpha_{1}
  \lneqq \widetilde{\lambda}_{1}(t),
\]
a contradiction. Hence $\widetilde{\lambda}_{1}(t) = \alpha_{1}$. Inductively,
we see $\widetilde{\lambda}_{j} = \alpha_{j}$. Thus
\[ %\tag{6.6} \label{e6.6}
  \lim_{r_{i} \searrow t} \widetilde{\lambda}_{j}(r_i)
  = \alpha_{j}
  = \widetilde{\lambda}_{j}(t),
\]
which shows the right-continuity.

\bigskip\noindent\emph{Step~$5$.} (Left-continuity) Let $\{u_{k}\}$ be the set
of eigenfunctions defined in Step~1. Given $k \in \mathbb{N}$. Consider $j =
1,2,\ldots,k$. By the first equality of Theorem~S, i.e., $\overline{\bigcup_{s
< t} H_{s}} = H_{t}$, there exist $s_{1}, s_{2}, \ldots \nearrow t$ and
$\overline{v}_{j}(s_{i}) \in H_{s_{i}} \subset H_{t}$, $\forall\, i =
1,2,\ldots$, such that $\overline{v}_{j}(s_{i}) \to u_{j}$ in $H_{t} \subset
H_{b}$ as $i \to \infty$. Using min-max principle~\eqref{e2.13}, we see that
\begin{align*}
  \widetilde{\lambda}_{k}(s_{i})
  &= \min_{V^{k} \subset H_{s_{i}}}
    \{ \max_{v \in V^{k} \cap S} \widetilde{I}(v,v) \}
    \qquad\qquad \textrm{(where $\dim V^{k} = k$)} \\
  &\leq \max_{v \in \langle \overline{v}_{1}, \ldots, \overline{v}_{k} \rangle
    \cap S} \: \widetilde{I}(v,v) \qquad\qquad\quad\;
    \textrm{(where $\overline{v}_{j} \equiv \overline{v}_{j}(s_{i})$)} \\
  &\xrightarrow{\textrm{as } i \to \infty}
    \max_{u \in \langle u_{1}, \ldots, u_{k} \rangle \cap S} \:
    \widetilde{I}(u,u)
    = \widetilde{\lambda}_{k}(t)
    \leq \widetilde{\lambda}_{k}(s_{i}), \quad \forall\, i,
\end{align*}
noting that $S$ is the unit sphere in $H_{s_i}$ or in $H_{t}$,  and
$\widetilde{I}_{s} = \widetilde{I}_{b}$, $\forall\, s \leq b$. Therefore
\[ %\tag{6.7} \label{e6.7}
  \lim_{s_{i} \nearrow t} \widetilde{\lambda}_{k}(s_{i})
  = \widetilde{\lambda}_{k}(t).
\]
The left-continuity is proved.

\bigskip\noindent\emph{Step~$6$.} (Strictness) It remains to show
\[
  s < t \;\; \Rightarrow \;\;
  \widetilde{\lambda}_{k}(s) \gneqq \widetilde{\lambda}_{k}(t),
    \;\forall\, k \in \mathbb{N}.
\]
Suppose $\exists\, k > 0$ with $\widetilde{\lambda}_{k}(s) =
\widetilde{\lambda}_{k}(t) \equiv \lambda$ for $s < t \leq b$. We shall
construct $u \in H_{s} \subset H_{t}$ such that
\[ %\tag{6.5} \label{e6.5}
  \widetilde{L}_{s}u = \lambda u \quad \textrm{and} \quad
  \widetilde{L}_{t}u = \lambda u,
\]
which is against $D(s) \subsetneqq D(t)$. Let $u_{1}, u_{2}, \ldots$ and
$v_{1}, v_{2}, \ldots$ be orthonormal bases of eigenfunctions of
$\widetilde{L}_{s}$ on $H_{s}$ and of $\widetilde{L}_{t}$ on $H_{t}$,
respectively. Choose $u = a_{1} u_{1} + \cdots + a_{k} u_{k} \in H_{s} \subset
H_{t}$, with $\|u\| = 1$, such that $\langle \: u,v_{1} \: \rangle = \cdots =
\langle \: u,v_{k-1} \: \rangle = 0$. Clearly, such $a_{i}$'s are solvable, and
$u \in C_{0}^{\infty}(D(s))$. Write
\[ %\tag{6.8} \label{e6.8}
  u = \langle \: u, v_{k} \: \rangle \: v_{k}
    + \langle \: u,v_{k+1} \: \rangle \: v_{k+1} + \cdots.
\]
Then
\begin{equation} \label{e4.4}
\begin{split}
  \langle \: \widetilde{L}_{t}u, u \: \rangle
  &= \widetilde{\lambda}_{k}(t) \: \langle \: u,v_{k} \: \rangle^{2}
    + \widetilde{\lambda}_{k+1}(t) \: \langle \: u, v_{k+1} \: \rangle^{2}
    + \cdots \\
  &\geq \widetilde{\lambda}_{k}(t) \:(\:\langle \: u,v_{k} \: \rangle^{2}
    + \langle \: u,v_{k+1} \: \rangle^{2} + \cdots \:)
  = \widetilde{\lambda}_{k}(t)
  = \lambda.
\end{split}
\end{equation}
But
\begin{equation} \label{e4.5}
\begin{split}
  \langle \: \widetilde{L}_{s}u, u \: \rangle
  &= \widetilde{\lambda}_{1}(s) a_{1}^{2} + \cdots
    + \widetilde{\lambda}_{k}(s) a_{k}^{2} \\
  &\leq \widetilde{\lambda}_{k}(s) \:(\: a_{1}^{2} + \cdots + a_{k}^{2} \:)
  = \widetilde{\lambda}_{k}(s)
  = \lambda.
\end{split}
\end{equation}
Given $w \in H_{t}$, we know
\[ %\tag{6.11} \label{e6.11}
  \langle \: \widetilde{L}_{s}u, w \: \rangle
  = \langle \: \widetilde{L}_{t}u, w \: \rangle - E
  \neq \langle \: \widetilde{L}_{t}u, w \: \rangle,
\]
in general, although $\widetilde{I}_{s}(u,w) = \widetilde{I}_{t}(u,w)$. Here
$E$ is the boundary term $\int_{\partial D(s)} (D_{\nu} u) w$. By choosing $w =
u \in H_{s} \subset H_{t}$, we have $w \equiv 0$ on $\partial D(s)$ and hence
$E = 0$. Thus
\begin{equation} \label{e4.6}
  \langle \: \widetilde{L}_{s}u, u \: \rangle
  = \langle \: \widetilde{L}_{t}u, u \: \rangle.
\end{equation}
A combination of \eqref{e4.4}, \eqref{e4.5} and \eqref{e4.6} yields that
$\lambda \geq \langle \: \widetilde{L}_{s}u, u \: \rangle = \langle \:
\widetilde{L}_{t}u, u \: \rangle \geq \lambda$. It forces all the inequalities
in \eqref{e4.4} and \eqref{e4.5} to be equalities. Hence
$\widetilde{\lambda}_{1}(s) = \widetilde{\lambda}_{2}(s) = \cdots =
\widetilde{\lambda}_{k}(s) = \lambda$, and $\widetilde{L}_{s}u = \lambda u$. On
the other hand, $\widetilde{\lambda}_{l}(t) \to \infty$ as $l \to \infty$.
Looking at \eqref{e4.4}, there is an $m \geq 0$ such that $u = \langle \:
u,v_{k} \: \rangle \: v_{k} + \cdots + \langle \: u,v_{k+m} \: \rangle \:
v_{k+m}$, and each of $v_{k}, \ldots, v_{k+m}$ has the same eigenvalue
$\lambda$. Thus
\[ %\tag{6.13} \label{e6.13}
  \widetilde{L}_{t}u = \lambda u.
\]
Now $u \in H_{s} \subset H_{t}$, $u \equiv 0$ on $D(t) - D(s)$. By the
assumption that $D(t) - D(s) \neq \phi$, there is a nonempty open set $U$
contained in $D(t) - D(s)$. Using Hopf's sphere theorem, we see that $u \equiv
0$ on $D(t)$, contradicting $\|u\| = 1$.
\end{proof}

As Theorem~B and Theorem~S have been established, the proof of the global
Morse index theorem, i.e., Theorem~A, is completed. Remark that we have
answered Problem~\ref{Prob2}, raised in the preface of the paper.

%%========================================================= sec. 5 ==
\section{Distribution of Jacobi fields} \label{S5}
%%===================================================================
We return to Problem~\ref{Prob1} raised in the preface of the paper, and
present a substantial application of the global Morse index theorem to detect
the existence of Jacobi fields on CMC surfaces.

%%============================================== sec. 6 -> sec. 5.1==
\subsection{Unstable cones} \label{S5.1}
%%===================================================================
Given a $C^{0}$-monotone continuum $\mathcal{D} = \{ D(t) \subset M^{n} \:;\: t
\in [0,b] \}$ defined by Definition~\ref{D2.7}, we consider the ambient space
$E_{b} \equiv E(D(b))$ of all Sobolev variation functions on $D(b)$, noting
that $E_{t} \equiv E(D(t)) \subset E_{b}$, $\forall\, t \leq b$, the inclusion
is defined in the second paragraph of \S\ref{S2.1}. Clearly, $H_{b} \equiv
H(D(b)) \subset E_{b}$ is a hyperplane of $E_{b}$, and $H_{t} \equiv H(D(t)) =
E_{t} \cap H_{b}$. For the unstable cones in the Sobolev framework, we define
for example $\Lambda(D) := \{ f \in E(D) \:;\: I(f,f) \leq 0 \}$ (see \eqref{e2.7}, \eqref{e2.8}). Let
$u_{1},u_{2},u_{3},\ldots$ be the eigenfunctions of $L$, which constitute an
orthonormal basis of $E(D)$, i.e., $Lu = \lambda_{k} u_{k}$, $k = 1,2,3,\ldots$
with $\lambda_{1} \leq \lambda_{2} \leq \lambda_{3} \leq \cdots \to \infty$.
Given $f \in E(D)$, we write  $f = a_{1} u_{1} + a_{2} u_{2} + \cdots$, then $f
\in \Lambda(D)$ if and only if
\begin{equation} \label{e5.1}
  \lambda_{1} a^{2}_{1} + \lambda_{2} a^{2}_{2} + \cdots \leq 0.
\end{equation}
The last inequality~\eqref{e5.1} together with $\int_{D} f = 0$ is the
criterion for $f \in \widetilde{\Lambda}(D) \equiv \Lambda(D) \cap H(D)$. We
see that $\Lambda(t) \equiv \Lambda(D(t)) \subset E_{t} \subset E_{b}$,
$\widetilde{\Lambda}(t) \equiv \Lambda(t) \cap H_{t} \subset H_{b}$. They
enlarge in $E_{b}$, as $t$ increases, i.e.,
\[ %\tag{3.2} \label{e3.2}
  s \leq t \;\; \Rightarrow \;\;
  \Lambda(s) \subset \Lambda(t) \;\; \textrm{and} \;\;
  \widetilde{\Lambda}(s) \subset \widetilde{\Lambda}(t).
\]
Let $t_{k}$ denote $t \in [0,b]$ for which $\lambda_{k} = \lambda_{k}(t) \equiv
\lambda_{k}(D(t)) = 0$. We also write $D(t_{k})$ by $D[\lambda_{k} = 0]$. For
$0 < t < t_{1}$, we have $\Lambda(t) = \{0\}$, since \eqref{e5.1} and
$\lambda_{k}(t) > 0$, $\forall\, k \geq 1$. As $t = t_{1}$, $D(t_{1}) =
D[\lambda_{1} = 0]$ is an extremal domain defined in \cite{HL98}, and
\[ %\tag{3.3} \label{e3.3}
  \Lambda(t_{1})
  = \textrm{ the linear span } \langle u_{1} \rangle
  = \textrm{ $1$-dim. subspace of } E_{t_{1}}
  \subset E_{b}.
\]
Evidently, $\Lambda(t_{1}) \cap H_{b} = {0}$, because the first eigenfunction
$u_{1}$ has constant sign on $D(t_{1})$ and hence $\int_{D(t_{1})} u_{1} \neq
0$ --- noting that $u_{k}$ are defined zero on $D(b) - D(t_{1})$, for $k =
1,2,3,\ldots$. When $t$ increases from $t_{1}$, $\Lambda(t)$ may touch $H_{b}$
at a (nontrivial) Jacobi field $\varphi \in H_{t} \subset H_{b}$ (see
Proposition~\ref{P1.7}). The existence of $\varphi$ can be shown by the
compactness of $\Lambda^{R}(t) := \{ f \in \Lambda(t) \:;\: \|f\|_{L^{2}} \leq
R \}$. However, a rigorous argument of the existence will be included later in
Theorem~J. Writing $t = c$ at the moment of contact, $D(c)$ is a
``critical" domain that $D(t)$ changes from being stable to unstable when $t$
increases across $c$ (see Theorem~\ref{T1.9}). The domain $D(c)$ has the
so-called ``first conjugate boundary" and there appears a (nontrivial) Jacobi
field on $D(c)$ (see Definition~\ref{D1.3}).

\begin{prop} \label{P5.1}
\begin{equation} \label{e5.2}
  D[\lambda_{1} = 0] \subsetneqq D(c) \subsetneqq D[\lambda_{2} = 0].
\end{equation}
\end{prop}
\begin{proof}
Evidently, $\widetilde{\Lambda}(t_{1}) = \{0\} \subsetneqq \langle \varphi
\rangle \subset \widetilde{\Lambda}(c)$, where $\varphi$ is the (nontrivial)
Jacobi field on $D(c)$. This shows the first “$\subsetneqq$” of \eqref{e5.2}.
We can find a nonzero $v = \alpha u_{1} + \beta u_{2} \in
\widetilde{\Lambda}_{-}(t_{2}) \subset H_{t_{2}} \subset H_{b}$. In fact,
$\int_{D} u_{1} \neq 0$ and hence we may choose $\beta = 1$, $\alpha =
-\int_{D} u_{2} \mathbin{/} \int_{D} u_{1}$, with $\int_{D} v = 0$ where $D =
D(t_{2})$. On the other hand, $\lambda_{1}(t_{2}) < 0 = \lambda_{2}(t_{2})$ on
$D(t_{2})$. Thus $I(v,v) = \lambda_{1} \alpha^{2} < 0$ and $v \in
\widetilde{\Lambda}_{-}(t_{2})$. Hence
\[ %\tag{3.5} \label{e3.5}
  \widetilde{\Lambda}_{-}(c)
  = \{0\}
  \subsetneqq \langle v \rangle
  \subset \widetilde{\Lambda}_{-}(t_{2}),
\]
which shows the second ``$\subsetneqq$".
\end{proof}

%%=========================================== sec. 3.2 -> sec. 5.2 ==
\subsection{Distribution theorem} \label{S5.2}
%%===================================================================
Consider the eigenvalues $\lambda_{j} = \lambda_{j}(t)$ of $L$ on $E_{t}$:
\begin{equation} \label{e5.3}
  \lambda_{1} < \lambda_{2} \leq \lambda_{3} \leq \cdots \leq \lambda_{l} < 0
  \leq \lambda_{l+1} \leq \cdots \to \infty.
\end{equation}
Relabel the indices $j$ of $\lambda_{j}$, such that the adjacent eigenvalues
are distinct, i.e., \eqref{e5.3} becomes
\begin{equation} \label{e5.4}
  \overline{\lambda}_{1} < \overline{\lambda}_{2} < \overline{\lambda}_{3}
  < \cdots < \overline{\lambda}_{k} < 0 \leq \overline{\lambda}_{k+1} < \cdots
  \to \infty,
\end{equation}
each $\overline{\lambda}_{h} = \overline{\lambda}_{h}(t)$ having multiplicity
$m_{h}$, for any $h \in \mathbb{N}$. Let
\begin{equation} \label{e5.5}
  n_{k} \equiv m_{1} + m_{2} + \cdots + m_{k}.
\end{equation}
Note that $l = n_{k}$, where $l$, $k$ and $n_{k}$ are respectively given in
\eqref{e5.3}, \eqref{e5.4} and \eqref{e5.5}. Let $\overline{t}_{h}$ denote $t
\in (0,b]$ such that $\overline{\lambda}_{h}(t) = 0$, i.e.,
\[ %\tag{3.5d} \label{e3.5d}
  D(\overline{t}_{h}) = D[\overline{\lambda}_{h} = 0],
    \quad \forall\, h \in \mathbb{N}.
\]
We ask whether Proposition~\ref{P5.1} extends to any adjacent
$\overline{\lambda}_{h-1}$ and $\overline{\lambda}_{h}$ for each $h > 1$?

\begin{thmJ}[Distribution of Jacobi fields along $t$-axis] \label{ThmJ}
Given $M^{n}$ a CMC hypersurface immersed in $\mathbb{R}^{n+1}$, let
$\mathcal{D} = \{ D(t) \subset M^{n} \:;\: t \in [0,b] \}$ be a
$C^{0}$-monotone continuum of generalized Lipschitz domains in $M^{n}$ (see
Definitions~\textup{\ref{D2.2}} and \textup{\ref{D2.7}}). For each $k > 1$,
there exist (nontrivial) Jacobi fields on $D(t)$ with some $t \in
[\overline{t}_{k-1},\overline{t}_{k}]$, i.e.,
\begin{equation} \label{e5.6}
  D[\overline{\lambda}_{k-1} = 0]
  \subset D(t)
  \subset D[\overline{\lambda}_{k} = 0].
\end{equation}
The multiplicity of the independent Jacobi fields on $D(t)$ satisfying
\eqref{e5.6} is given as follows. Let $\mu(T)$ denote the total number of
independent Jacobi fields on $D(t)$ with $t \in T \subset [0,b]$. Then,
\begin{equation} \label{e5.7}
  1
  \leq m_{k-1} + m_{k} - 1
  \leq \mu[\overline{t}_{k-1},\overline{t}_{k}]
  \leq m_{k-1} + m_{k} + 1,
\end{equation}
where $m_{k}$ is the multiplicity of $\overline{\lambda}_{k}$ and $T$ is the
closed interval $[\overline{t}_{k-1},\overline{t}_{k}]$. Furthermore, we have
\begin{gather}
\label{e5.8}
  m_{k}-1 \leq \mu(\overline{t}_{k-1},\overline{t}_{k}] \leq m_{k}+1, \\
\label{e5.9}
  m_{k}-1 \leq \mu \{\overline{t}_{k}\} \leq m_{k}+1.
\end{gather}
\end{thmJ}

\begin{rmk} \label{R5.2}
The statement~\eqref{e5.6} of Theorem~J is still valid, even when the adjacent
eigenvalues are not distinct, i.e., when $t_{h-1} = t_{h}$. Namely, let
$\{ \overline{\lambda}_{k}, \overline{t}_{k} \}$ be replaced by $\{
\lambda_{h}, t_{h} \}$ of \eqref{e5.3} with $h > 1$, there exist independent
Jacobi fields on $D(t)$ such that
\begin{equation} \label{e5.10}
  D[\lambda_{h-1} = 0] \subset D(t) \subset D[\lambda_{h} = 0]
\end{equation}
with multiplicity at least $m-1$, whether $t_{h-1}$ and $t_{h}$ are distinct or
not. Here $m$ is the multiplicity of $\lambda_{h} = 0$. For example, if
$\lambda_{h} = 0$ has multiplicity $5$, then there are at least $4$ independent
Jacobi fields on $D(t)$ with $D(t) = D[\lambda_{h-1} = 0] = D[\lambda_{h} =
0]$. The assertion~\eqref{e5.10} follows from \eqref{e5.9}, in which $m_{k}-1 =
m-1 > 0$, and from \eqref{e5.7}.
\end{rmk}

\begin{ex} \label{E5.1}
Consider the simplest case, that $\dim M^{n} = n = 1$. For $x \in \mathbb{R}$,
let $\varphi = \varphi(x) = e^{ix} \in$ the unit circle $S^{1} \subset
\mathbb{C}$ define an immersion from $\mathbb{R}^{1}$ into $S^{1}$. Consider on
$M^{1} \equiv \mathbb{R}^{1}$ a $C^0$-monotone continuum $\mathcal{D}_{1} = \{
D(t) \subset M^{1} \:;\: t \in (0,\infty) \}$, where each $D(t)$ is the open
interval $(0,t)$ in $\mathbb{R}^{1}$.

\medskip\noindent (i) Find $t = \overline{t}_{k}$ such that
$D(\overline{t}_{k}) = D[\overline{\lambda}_{k} = 0]$ for some $k > 1$.
Equation~\eqref{e2.7} is reduced to
\[ %\tag{3.12} \label{e3.12}
  Lf = -f'' - f, \quad f = f(x), \quad x \in D(t), \; f \in E(D(t)),
\]
where $|B|^{2} = 1$ on $S^{1}$, and \eqref{e1.3} means now
\[ %\tag{3.13} \label{e3.13}
  u''_{k} = -(1+\overline{\lambda}_{k}) u_{k}, \quad k = 1,2,3,\ldots
\]
with $u_{k} \in \mathcal{F}(D(t))$. It is easy to see that for a fixed $t$, the
eigenfunction $u_{k}$ on $D(t)$ is given by
\[ %\tag{3.14} \label{e3.14}
  u_{k}(x) = c_{k} \sin \left( \frac{k\pi}{t} \right) x,
  \quad \textrm{and} \quad
  \overline{\lambda}_{k} = \frac{k^{2} \pi^{2}}{t^{2}} - 1.
\]
Hence $m_{k} = 1$, $\overline{\lambda}_{k} = \lambda_{k}$ and $\overline{t}_{k}
= t_{k}$. Thus, $\lambda_{k} = 0$ if and only if $t = k\pi$, i.e., $t_{k} =
k\pi$ and $D(k\pi) = D[\lambda_{k} = 0]$. Note that $D(\pi)$ is an extremal
domain on which $\lambda_{1} = 0$.

\medskip\noindent (ii) Clearly, $g \in \mathcal{G}(D(t))$ is a Jacobi field on
$D(t)$ if and only if
\[ %\tag{3.15} \label{e3.15}
  Lg = -g''-g = c \quad
    \textrm{for some constant $c = \frac{1}{|D|} \int_{D} Lg$}.
\]
Here $D = D(t)$. The general solution is $g(x) = A\cos x + B\sin x - c$ with
\[ %\tag{3.16} \label{e3.16}
  g(0) = 0 = g(t), \quad \int_{0}^{t} g(x) \, dx = 0.
\]
A straightforward calculation shows that $D(t)$ has a (nontrivial) Jacobi field
if and only if
\[ %\tag{3.17} \label{e3.17}
  \psi(t) := 2 - 2\cos t - t\sin t = 0.
\]

\medskip\noindent (iii) Sketching the graph of $\psi$ (see Figure~\ref{F13}),
we find the zeros $l_{1}, l_{2}, l_{3}, \ldots$ of $\psi$ as follows:
\[ %\tag{3.18} \label{e3.18}
\begin{aligned}
  l_{1} &= 2\pi \in (\pi,2\pi], &\quad
  l_{2} &\in (5\pi/2,3\pi) \subset (2\pi,3\pi] \\
  l_{3} &= 4\pi \in (3\pi,4\pi], &\quad
  l_{4} &\in (9\pi/2,5\pi) \subset (4\pi,5\pi] \\
  l_{5} &= 6\pi \in (5\pi,6\pi], &\quad
  l_{6} &\in (13\pi/2,7\pi) \subset (6\pi,7\pi] \\
  l_{7} &= 8\pi \in \cdots.
\end{aligned}
\]
Therefore, $\varphi(0,l_{k})$ are exactly the domains on which some
(nontrivial) Jacobi fields exist. This one-dimensional example shows that
\[ %\tag{3.19} \label{e3.19}
  t_{k} < l_{k} \leq t_{k+1}
\]
which provides an evidence supporting Theorem~J.
\begin{figure}[H]
\centering
\includegraphics[width=0.7\textwidth]{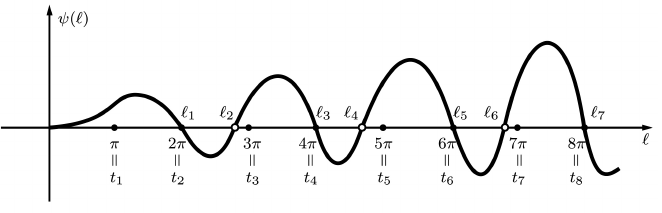}
\caption{} \label{F13}
\end{figure}
\end{ex}

%%============================================ sec. 3.3 -> sec. 5.3==
\subsection{Proof of Theorem~J} \label{S5.3}
%%===================================================================
The idea of proving Theorem~J is sketched in the following steps.
\begin{enumerate}
\item[\textbf{J1}] Establish the continuity of $\lambda_{k}(t)$ and
$\widetilde{\lambda}_{k}(t)$ in $t$, which has been proved in Theorem~B.

\item[\textbf{J2}] Observe the relationship between
$\Lambda_{-}(t)$ and $\widetilde{\Lambda}_{-}(t)$ and show in the following
Lemma~D that
\[ %\tag{3.9} \label{e3.9}
  i(t)-1 \leq \widetilde{i}(t) \leq i(t).
\]

\item[\textbf{J3}] Use the global Morse index theorem (Theorem~A) to find the
number of independent Jacobi fields on $D(t)$, based on the estimate of
$\widetilde{i}(t)$ in J2.
\end{enumerate}

\begin{lemD} \label{LemD}
Given a $C^{0}$-monotone continuum $\mathcal{D} = \{ D(t) \subset M^{n} \:;\: t
\in [0,b] \}$ of generalized Lipschitz domains, let $i(t)$ and $\widetilde{i}(t)$ be Morse
indices of $L$ on $E_{t}$ and $\widetilde{L}$ on $H_{t}$ respectively. Then
(see Figure~\textup{\ref{F14}})
\begin{equation} \label{e5.11}
  i(t)-1 \leq \widetilde{i}(t) \leq i(t).
\end{equation}
\begin{figure}[H]
\centering
\includegraphics{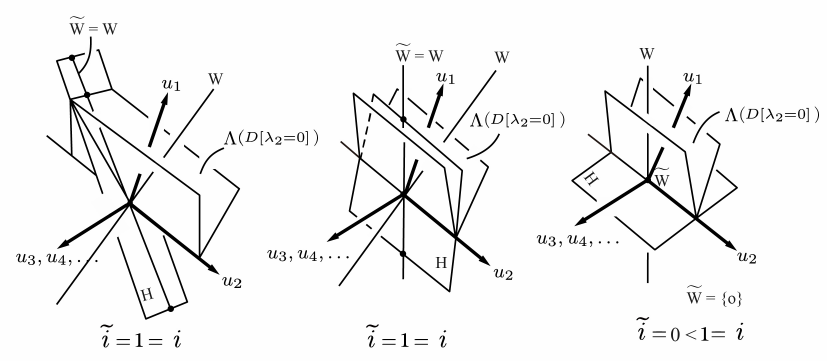}
\caption{} \label{F14}
\end{figure}
\end{lemD}
\begin{proof}
\emph{Step~$1$.} Given $t \in (0,b]$, let $W$ and $\widetilde{W}$ be two
maximal linear subspaces contained in $\Lambda_{-}(t)$ and
$\widetilde{\Lambda}_{-}(t) \equiv \Lambda_{-}(t) \cap H_{t}$, respectively.
Evidently, $\widetilde{W} \subset W$, and hence $\widetilde{i}(t) \leq i(t)$.
By definition, $\Lambda_{-}(t)$ consists of $0$ and all $f = a_{1} u_{1} +
a_{2} u_{2} + \cdots$ in $E_{t}$ such that
\begin{equation} \label{e5.12}
  I(f)
  \equiv I(f,f)
  = \lambda_{1} a_{1}^{2} + \lambda_{2} a_{2}^{2} + \cdots
  < 0.
\end{equation}
Let $l \in \mathbb{N}$ be such that
\begin{equation} \label{e5.13}
  \lambda_{1} < \lambda_{2} \leq \lambda_{3} \leq \cdots \leq \lambda_{l} < 0
  \leq \lambda_{l+1} \leq \cdots \to \infty,
\end{equation}
where $\lambda_{j} = \lambda_{j}(t)$, $\forall\, j \in \mathbb{N}$. Consider
$W_{l} :=$ the linear span $\langle u_{1}, u_{2}, \ldots, u_{l} \rangle$. Then
$W_{l} \subset \Lambda_{-}(t)$ by \eqref{e5.12} and \eqref{e5.13}.

\bigskip\noindent\emph{Step~$2$.} Claim that $W_{l}$ is maximal. Suppose not,
then there exists an $h = h_{1} u_{1} + h_{2} u_{2} + \cdots \notin W_{l}$ with
$g+h \in \Lambda_{-}(t)$, $\forall\, g \in W_{l}$. Select $g = -h_{1} u_{1} -
h_{2} u_{2} - \cdots - h_{l} u_{l}$, then $g+h \neq 0$ and $I(g+h) =
\lambda_{l+1} u_{1}^{2} + \cdots > 0$, a contradiction. Therefore, $i(t) = \dim
W_{l} = l$.

\bigskip\noindent\emph{Step~$3$.} Remark that $u_{1} > 0$ on $D(t)$ and
$\int_{D(t)} u_{1} \neq 0$. Hence, $u_{1} \notin H_{t}$ and $W_{l} \not\subset
H_{t}$, since $u_{1} \in W_{l}$. Construct a hyperplane $V$ of $W_{l}$ such
that $V \subset \Lambda_{-}(t)$ as follows. Consider $g = a_{1} u_{1} + \cdots
+ a_{l} u_{l}$, where $a_{2}, \ldots, a_{l}$ are arbitrary constants and
\[ %\tag{4.4} \label{e4.4}
  a_{1} = \textstyle{(-a_{2} \int_{D} u_{2} - \cdots - a_{l} \int_{D} u_{l})
  \:\big/\: \int_{D} u_{1}}
\]
with $D = D(t)$. Then $\int_{D(t)} g = 0$ and $g \in H_{t}$. Let $V$ be the
$(l-1)$-dimensional linear subspace consisting of all such $g$. Then $V \subset
\widetilde{\Lambda}_{-}(t)$, and therefore $\widetilde{i}(t) \geq \dim V = l-1
= i(t)-1$.
\end{proof}

\begin{rmk} \label{R5.3}
For the relationship of $\Lambda_{-}(t)$ and $H_{t}$, there are three geometric
types. In Figures~\ref{F14}(a) and \ref{F14}(b), $\widetilde{i}(t) = 1 = i(t)$,
while in Figure~\ref{F14}(c), $\widetilde{i}(t) = 0 < 1 = i(t)$.
\end{rmk}

%%=========================================== sec. 4.2 -> sec. 5.3 ==
It remains to deal with J3 to finish the proof of Theorem~J.

\begin{proof}[Proof of Theorem~\textup{J}]
\emph{Step~$1$.} Recall that we have relabelled in \eqref{e5.4} the indices $j$
of the eigenvalues $\lambda_{j}$ of $L$ by $\overline{\lambda}_{k}$ with
multiplicities $m_{k}$, such that the adjacent eigenvalues are distinct. Let
$\overline{t}_{k}$ denote $t \in (0,b]$ such that $D(\overline{t}_{k}) =
D[\overline{\lambda}_{k} = 0]$. Denote by $\overline{t}^{+}_{k}$ for such $t$
that is sufficiently close to $\overline{t}_{k}$ from the right, i.e.,
$\overline{t}_{k} < \overline{t}^{+}_{k}$ and $\overline{t}^{+}_{k} -
\overline{t}_{k}$ is sufficiently small. Consider \eqref{e5.13} where
$\lambda_{j} = \lambda_{j}(\overline{t}^{+}_{k})$. Assume Theorem~B which
states that $\lambda_{j}(t)$ decreases strictly and continuously in $t$. As $t$
increases from $0$ to $\overline{t}^{+}_{k}$, those $j$ for which
$\lambda_{j}(t)$ pass through zero and become negative are exactly $j =
1,2,\ldots,l$ with
\[ %\tag{4.5} \label{e4.5}
  l = m_{1} + m_{2} + \cdots + m_{k} = n_{k}.
\]
Thus $i(\overline{t}^{+}_{k}) = n_{k}$, and by \eqref{e5.11} in Lemma~D,
\[ %\tag{4.6} \label{e4.6}
  n_{k}-1 \leq \widetilde{i}(\overline{t}^{+}_{k}) \leq n_{k}.
\]

\bigskip\noindent\emph{Step~$2$.} Applying the Morse index theorem (Theorem~A),
we see that
\begin{equation} \label{e5.14}
  n_{k}-1 \leq \mu(0,\overline{t}_{k}] \leq n_{k}
\end{equation}
since $\mu(0,\overline{t}_{k}] = \sum_{0 < t \leq \overline{t}_{k}}
\widetilde{\nu}(t) = \widetilde{i}(\overline{t}^{+}_{k})$. But
\[ %\tag{4.8} \label{e4.8}
  \mu(\overline{t}_{k-1},\overline{t}_{k}]
  = \mu(0,\overline{t}_{k}] - \mu(0,\overline{t}_{k-1}].
\]
Using \eqref{e5.14}, we obtain
\[ %\tag{4.9} \label{e4.9}
  m_{k}-1 \leq \mu(\overline{t}_{k-1},\overline{t}_{k}] \leq m_{k}+1,
\]
noting that $m_{k} = n_{k} - n_{k-1}$. This shows \eqref{e5.8} in Theorem~J. By
considering $\mu[\overline{t}_{k-1}, \overline{t}_{k}] =
\mu(0,\overline{t}_{k}] - \mu(0,\overline{t}_{k-1}^{-}]$, where
$\overline{t}_{k-1}^{-}< \overline{t}_{k-1}$ is a number sufficiently close to
$\overline{t}_{k-1}$, we obtain \eqref{e5.7} with the similar argument as
above. Hence \eqref{e5.6} is established.

\bigskip\noindent\emph{Step~$3$.}
To show \eqref{e5.9}, we note that
\[ %\tag{4.11} \label{e4.11}
  \mu \{\overline{t}_{k}\}
  = \mu(0,\overline{t}_{k}] - \mu(0,\overline{t}^{-}_{k}].
\]
By the similar argument of Step~2, it is obtained that
\[ %\tag{4.12} \label{e4.12}
  m_{k}-1 \leq \mu \{\overline{t}_{k}\} \leq m_{k}+1.
\]
The proof of Theorem~J is completed.
\end{proof}

Theorem~J answers Problem~\ref{Prob1} given in the preface.

%%============================================================ Ref ==

%%================================ Authors' Affiliations & E-mails ==
\begin{flushleft}
Wu-Hsiung Huang \\
Department of Mathematics, National Taiwan University, Taipei 106, Taiwan \\
\emph{E-mail address}: \texttt{whuang0706@gmail.com}
\end{flushleft}
\end{document}